\documentclass[reqno]{amsart}

\usepackage[utf8]{inputenc}
\usepackage[OT2,T1]{fontenc}
\usepackage[english]{babel}

\usepackage{amsmath}
\usepackage{amsfonts}
\usepackage{amssymb}
\usepackage{amsthm}
\usepackage{mathrsfs}
\usepackage{listings}
\usepackage[mathcal]{euscript}

\usepackage{url}
\usepackage{hyperref}

\usepackage{multirow}

\usepackage[dvipsnames]{xcolor}
\usepackage{graphicx}
\usepackage{tikz-cd}
\usepackage[all]{xy}

\usepackage{subcaption}
\usepackage{colortbl}
\definecolor{mygray}{gray}{0.92}
\newcommand{\ctbl}{\cellcolor{mygray}}

\usepackage{array}
\newcolumntype{C}[1]{>{\centering\arraybackslash$}p{#1}<{$}}

\theoremstyle{plain}
\newtheorem{theorem}{Theorem}[section]

\newtheorem{proposition}[theorem]{Proposition}
\newtheorem{lemma}[theorem]{Lemma}

\theoremstyle{definition}
\newtheorem{definition}[theorem]{Definition}

\theoremstyle{remark}
\newtheorem{remark}[theorem]{Remark}
\newtheorem{example}[theorem]{Example}

\numberwithin{equation}{section}

\def\defi{\textsf}

\def\eps{\varepsilon}
\def\epsilon{\varepsilon}
\def\theta{\vartheta}
\def\phi{\varphi}
\def\tilde{\widetilde}
\newcommand{\trv}[3]{( #1,#2 )_{#3}}

\def\LiE{\textsc{LiE}}
\def\Magma{\textsc{Magma}}

\def\cf{\textit{cf.}}

\def\ie{\textit{i.e.}}
\def\infra{\textit{infra}}

\DeclareMathOperator{\disc}{disc}

\DeclareMathOperator{\DO}{DO}

\DeclareMathOperator{\GL}{GL}
\DeclareMathOperator{\GO}{GO}

\DeclareMathOperator{\Hom}{Hom}
\DeclareMathOperator{\im}{im}

\DeclareMathOperator{\intt}{int}

\DeclareMathOperator{\JI}{JI}

\DeclareMathOperator{\PSL}{PSL}

\DeclareMathOperator{\SL}{SL}
\DeclareMathOperator{\SO}{SO}

\DeclareMathOperator{\Spec}{Spec}

\DeclareMathOperator{\Sym}{Sym}

\def\A{\mathbb{A}}
\def\C{\mathbb{C}}
\def\F{\mathbb{F}}

\def\N{\mathbb{N}}

\def\P{\mathbb{P}\,}
\def\Q{\mathbb{Q}}
\def\R{\mathbb{R}}

\def\Z{\mathbb{Z}}

\def\el{\mathfrak{e}}
\def\fl{\mathfrak{f}}
\def\hl{\mathfrak{h}}
\def\sll{\mathfrak{sl}}

\usepackage{bm}
\def\gammab{\bm{\gamma}}
\def\Deltab{\bm{\Delta}}
\def\etab{\bm{\eta}}
\def\rhob{\bm{\rho}}
\def\sigmab{\bm{\sigma}}

\def\psib{\bm{\psi}}
\def\chib{\bm{\chi}}
\def\Cb{\bm{C}}
\def\Fb{\bm{F}}
\def\Hb{\bm{H}}
\def\Ib{\bm{I}}
\def\Jb{\bm{J}}
\def\Nb{\bm{N}}
\def\Pb{\bm{P}}
\def\Qb{\bm{Q}}
\def\Tb{\bm{T}}
\def\Xb{\bm{X}}
\def\bb{\bm{b}}
\def\fb{\bm{f}}
\def\gb{\bm{g}}
\def\hb{\bm{h}}
\def\jb{\bm{j}}

\def\ub{\bm{u}}

\def\Xs{\mathsf{X}}
\def\Ys{\mathsf{Y}}
\def\Zs{\mathsf{Z}}

\def\av{\underline{a}}
\def\bv{\underline{b}}
\def\bbv{\underline{\bm{b}}}
\def\Iv{\underline{I}}
\def\vv{\underline{v}}

\def\xv{\underline{x}}

\hypersetup{
  pdfauthor   = {Reynald Lercier, Christophe Ritzenthaler, Jeroen Sijsling}
  pdftitle    = {Reconstructing plane quartics from their invariants}
  pdfsubject  = {},
  pdfkeywords = {},
  backref=true, pagebackref=true, hyperindex=true, colorlinks=true,
  breaklinks=true, urlcolor=blue, linkcolor=blue, citecolor=blue,
  bookmarks=true, bookmarksopen=true}

\begin{document}

\title{Reconstructing plane quartics from their invariants}
\date{\today}

\begin{abstract}
  We present an explicit method that, given a generic tuple of Dixmier--Ohno
  invariants, reconstructs a corresponding plane quartic curve.
\end{abstract}

\author{Reynald Lercier}
\address{%
  \textsc{DGA MI}, %
  La Roche Marguerite, %
  35174 Bruz, %
  France. %
}
\address{%
  IRMAR, %
  Universit\'e de Rennes 1, %
  Campus de Beaulieu, %
  35042 Rennes, %
  France. %
}
\email{reynald.lercier@m4x.org}

\author{Christophe Ritzenthaler}
\address{%
  IRMAR, %
  Universit\'e de Rennes 1, %
  Campus de Beaulieu, %
  35042 Rennes, %
  France. %
}
\email{christophe.ritzenthaler@univ-rennes1.fr}

\author{Jeroen Sijsling}
\address{
  Department of Mathematics,
  Dartmouth College,
  6188 Kemeny Hall,
  Hanover, NH 03755
  USA
}
\address{
  Institut f\"ur Reine Mathematik,
  Universit\"at Ulm,
  Helmholtzstrasse 18,
  89081 Ulm,
  Germany
}
\email{sijsling@gmail.com}

\thanks{The first two authors acknowledge support from the CysMoLog ``d\'efi
  scientifique \'emergent'' of the Universit\'e de Rennes 1.}

\subjclass[2010]{13A50, 14L24, 14H10, 14H25}
\keywords{plane quartic curves; invariant theory; Dixmier--Ohno invariants;
  moduli spaces; reconstruction}

\maketitle

\section*{Introduction}

Invariant theory played a central role in nineteenth-century algebra and
geometry, and the natural action of linear groups on spaces of homogenous
polynomials in several variables, or \defi{forms}, was one of its principal
areas of focus. Nowadays, motivated by computational applications to
cryptography, robotics, coding theory, and experimental mathematics, this
theory has come to a renaissance. It is a source of many questions with an
explicit or computational orientation. One of these questions, the
reconstruction of ternary quartics (\ie, forms of degree $4$ in $3$ variables)
from their invariants, is the central theme of our paper.

Of old, the group $\SL_2 (\C)$ and its action on the ring $R_{2,n}$ of binary
forms $b (z_1, z_2)$ of degree $n \geq 2$ over $\C$ have received the lion's
share of attention. One reason is the remarkable formalism developed by Gordan
in 1868~\cite{gordan68} and by Hilbert in 1897~\cite{hilbert-inv} to compute a
finite set of generators of the ring of invariants $\C [R_{2,n}]^{\SL_2 (\C)}$
of $R_{2,n}$. The implementations of their approaches have so far led to the
determination of a set of generators (called \defi{fundamental invariants}) of
these invariant rings for $n \leq 10$; here we refer
to~\cite{brouwer2,brouwer1,gordan68,gall80,vongall}.

This algebraic problem has a geometric counterpart. Given a binary form $b$ of
even degree $n$ with simple roots, we can consider the degree $2$ cover $\Xs$
of $\P^1_{\C}$ defined by the equation $\Xs : y^2 = b (z_1, z_2)$ in the
weighted projective space with coordinates $z_1, z_2, y$ and weights $1, 1, n$.
Since isomorphisms between hyperelliptic curves are induced by the action of
$\GL_2(\C)$, the values of fundamental invariants of $R_{2,n}$ on the form $b$
define a point in a certain weighted projective space which characterizes the
geometric isomorphism class of $\Xs$. For $n = 4$ an affine coordinate on this
space is given by the classical Weierstrass $j$-invariant; for larger even $n$
this construction gives rise to an explicit embedding of the moduli space of
hyperelliptic curves of genus $n / 2 - 1$ into a weighted projective space.

When we broaden our scope to consider the action of $\SL_3(\C)$ on the ring
$R_{3,n}$ of ternary homogenous polynomials $F$ of degree $n$, explicit sets of
generators are known only for $n \leq 4$. While the cases $n = 2, 3$ were known
classically, the case $n = 4$ was completely settled only recently, by
Dixmier~\cite{dixmier} and Ohno~\cite{ohno} (but see also~\cite{elsenhans,
  giko}). The work of these authors shows that the ring $\C [R_{3,4}]^{\SL_3
  (\C)}$ is generated by $13$ elements, usually called the \defi{Dixmier--Ohno
  invariants} of ternary quartics. We discuss these results in
Section~\ref{sec:dixmier}.

The cases $n = 2, 3$ correspond to curves of genus $0$ and $1$, which can at
least equally conveniently be seen as degree $2$ covers of the projective line.
Upon passing to the case $n = 4$, this changes drastically. The associated
plane quartic curve $\Xs : F (x_1, x_2, x_3) = 0$ is now a
\emph{non}-hyperelliptic curve of genus $3$. Two plane quartics that are smooth
(as is generically the case) are isomorphic if and only if the corresponding
ternary quartic forms are equivalent under the action of $\GL_3 (\C)$. \medskip

In this algebraic-geometric setup
\begin{eqnarray*}
  \{ \textrm{binary/ternary forms}\}_{/\GL}
  & \longleftrightarrow
  & \{\textrm{classical invariants}\} \\
  \left\{ \textrm{curves of genus} \; g \right\}_{/\simeq}
  & \longleftrightarrow
  & \left\{\textrm{space of coordinates}\right\}
\end{eqnarray*}
we have so far only indicated how to go from left to right. However, there are
many areas where one is also interested in the reconstruction of a curve from
given invariants. Here one can think of the construction of CM
curves~\cite{cohen-handbook} but also of finding curves with many
points~\cite{rokaeus-search} or of experiments in arithmetic
statistics~\cite{LRRS}. For binary forms of degrees $n \in \left\{ 4,6,8
\right\}$, that is to say, for elliptic curves and for hyperelliptic curves of
genus $\le 3$, this reconstruction is indeed possible over any algebraically
closed field of characteristic $p$ for $n=4$ and $6$ and for fields of
characteristic $0$ or $p > 7$ if $n=8$. The cases $2 \leq p \leq 7$ were
analyzed in~\cite{basson}; reconstruction turns out to be possible for $p \in
\left\{ 3, 7 \right\}$, and for $p = 2$ an orbit-separating set of invariants
can be constructed.

The main tool in this context is a method due to Mestre~\cite{mestre} (see the
introduction of~\cite{LR11} for more details). It is based on formulas that go
back to Clebsch~\cite[\S~103]{clebsch} and uses a generalization of invariants
called \defi{covariants} (see Definition~\ref{def:variant}). Roughly speaking,
starting from three covariants of order $2$, one constructs a conic $\Qb$ and a
plane curve $\Hb$ of degree $n / 2$. The coefficients of the curves $\Qb$ and
$\Hb$ are invariants, and can in particular be expressed in terms of the
fundamental invariants mentioned above. This means that when given values of
the fundamental invariants, we can construct the corresponding specializations
$Q$ and $H$ of $\Qb$ and $\Hb$. If the resulting specialization $Q$ is
non-singular, then we can build the degree $2$ cover of $Q$ that ramifies at
the $n$ points where $Q$ and $H$ intersect. This cover will then have the
requested fundamental invariants. Note that the singularity of the conic $\Qb$
is in turn described by the vanishing of a generically non-zero invariant.
Therefore this method furnishes us with a way to reconstruct generic
hyperelliptic curves of genus at most $3$ from their invariants. When using
enough distinct covariants, this reconstruction in fact becomes possible for
any hyperelliptic curve of genus at most $3$, \cf~\cite{LR11}.

For ternary forms, the reconstruction problem is more complicated. Of course in
low degree the obstacles are not very extensive yet. Certainly the case of
ternary quadrics is manageable enough, as all of these are geometrically
isomorphic to $\P^1$. The case of ternary cubics can be dealt with by
considering these genus $1$ curves as an elliptic curve and reconstructing this
curve from its Weierstrass $j$-invariant (see~\cite[4.5, p.173]{sturmfels} and
\cite[Sec.10.3]{dolgainv}). The case of ternary quartics, however, has been
open for a long time. This is in part because the formulas by Clebsch that are
used in Mestre's method remain mysterious, in the sense that no general
encompassing formalism for their derivation has been found yet. Another
complicating factor in the ternary quartic case is the high dimension of the
underlying geometry of the space of ternary quartics. Of course,
Katsylo~\cite{katsylo-bi} showed that the moduli space of plane quartics is
rational. As such, by the triviality of the descent obstruction for generic
plane quartics there should exist a quartic curve over the purely
transcendental field $\C (\P \Sym^4 (V^*))^{\SL_3 (\C)}$ whose coefficients are
functions in the Dixmier--Ohno invariants. However, the generators of this
field are not known explicitly, and if the binary case~\cite{maeda90} is a fair
indication, the corresponding invariants are likely to be rather unmanageable,
which leaves even less hope for writing down a such a generic quartic curve.

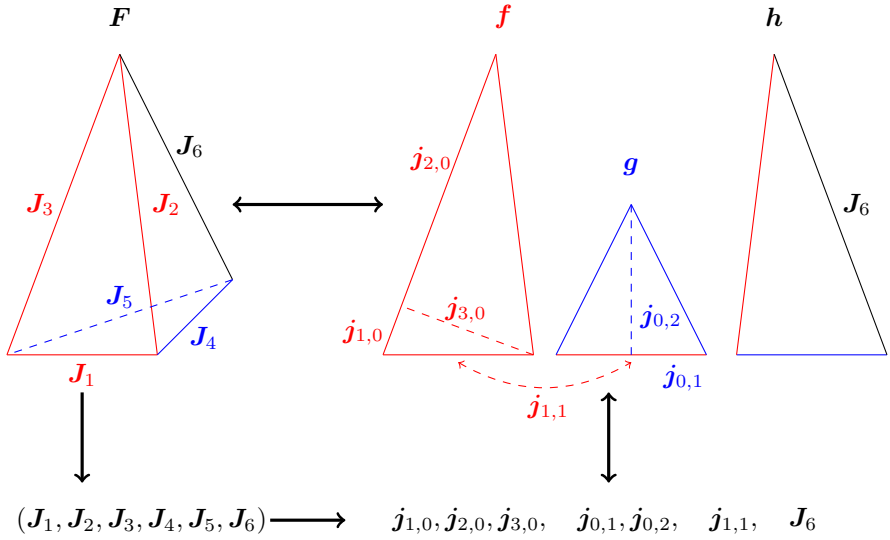
\begin{figure}[htbp]\centering
  \begin{center}
    \begin{tikzpicture}
      \draw [color=red] (0,3)-- (2,3);
      \draw [color=red] (1,3) node[below] {$\Jb_1$};
      \draw [color=red] (2,3)-- (1.5,7);
      \draw [color=red] (2,5) node {\ \ $\Jb_2$};
      \draw  [color=red] (1.5,7)-- (0,3);
      \draw [color=red] (0.5,5) node {$\Jb_3$\ \ };
      \draw (1.5,7) -- (3,4);
      \draw (2,5.5) node[above right] {\ $\Jb_6$};
      \draw [color=blue] (3,4) -- (2,3);
      \draw [color=blue] (2.5,3.5) node[below] {$\ \ \Jb_4$};
      \draw  [color=blue] [dashed] (3,4)--(0,3);
      \draw  [color=blue] (1.5,3.5) node[above] {$\Jb_5$};
      \draw (1.5,7.5) node {$\Fb$};

      \draw [very thick] [<->] (3,5) -- (5,5);
      \draw [very thick] [->] (1,2.5) -- (1,1.3);

      \draw [color= red] (5,3)-- (7,3);
      \draw [color= red] (5.7,5.6) node {$\jb_{2,0}$\ \ };
      \draw [color= red] (4.8,3.3) node {$\jb_{1,0}$\ \ };
      \draw  [color= red] (7,3)-- (6.5,7);
      \draw  [color= red] (6.5,7)-- (5,3);
      \draw [color= red] [dashed] (7,3) -- (5.24,3.65);
      \draw [color= red] (6.5,3.6) node[left] {$\jb_{3,0}$};
      \draw [color= red] (6.6,7.5) node {$\fb$};

      \draw [color=red, dashed] [<->] (6,2.9) to [bend right] (8.3,2.9);
      \draw [color= red] (7.2,2.3) node {$\jb_{1,1}$};

      \draw [color= red] (7.3,3)-- (9.3,3);
      \draw [color= blue] (9,2.7) node {$\jb_{0,1}$};
      \draw [color= blue] [dashed] (8.3,5) -- (8.3,3);
      \draw [color= blue] (8.3,3.5) node[right] {$\jb_{0,2}$};
      \draw [color=blue] (9.3,3)-- (8.3,5);
      \draw [color=blue] (8.3,5)-- (7.3,3);
      \draw [color= blue] (8.3,5.5) node {$\gb$};

      \draw  [color=blue] (9.7,3) -- (11.7,3);
      \draw [color=red] (9.7,3) -- (10.2,7);
      \draw (10.2,7) -- (11.7,3);
      \draw (11.3,5) node {$\Jb_6$};
      \draw (10.2,7.5) node {$\hb$};
      \draw [very thick] [<->] (8,2.5) -- (8,1.3);

      \draw (0,0.5) node[above right] {$(\Jb_1,\Jb_2,\Jb_3,\Jb_4,\Jb_5,\Jb_6)$};
      \draw [very thick] [->] (3.5,.8) -- (4.5,.8);
      \draw (5,0.5) node[above right] {$\jb_{1,0},\jb_{2,0},\jb_{3,0}, \quad \jb_{0,1},\jb_{0,2}, \quad \jb_{1,1}, \quad \Jb_6$};

    \end{tikzpicture}

  \end{center}
  \caption{Geometric analogue of the reconstruction strategy}
  \label{fig:analogy}
\end{figure}

Our strategy therefore takes a detour and involves several steps. We illustrate
them in Figure~\ref{fig:analogy}.  This picture is to be interpreted as
follows. On the left hand side, we consider a ``universal'' tetrahedron $\Fb$;
this universality is reflected in our notation with a boldface, as for the
``universal'' curves $\Qb$ and $\Hb$ considered in Mestre's method above. This
tetrahedron $\Fb$ is the analogue of the ternary quartic forms in which we are
interested. The lengths of the sides of $\Fb$ correspond to the invariants of
this form. Note the analogy: the side lengths of a tetrahedron are indeed
invariant under isometry. We let $\Jb_i$ be a set of side lengths from which
our tetrahedron $\Fb$ can be reconstructed; these are the analogue of the
Dixmier--Ohno invariants.

We reconstruct $\Fb$ from the $\Jb_i$ by an indirect route by taking a
counterclockwise detour through a composition of three other arrows, which are
defined in the following way. The top arrow is the extraction of three of the
faces $\fb, \gb, \hb$ of $\Fb$ and conversely, we assume that we have an
explicit  way to reconstruct $\Fb$ from these faces. Moreover, these faces are
in turn characterized by certain lengths $\jb_{0,b}$ and $\jb_{a,0}$, and these
lengths additionally satisfy some mutual relations. The face length $\jb_{1,1}$
is the analogue of what we will later call a \defi{joint invariant}. It is
easier to reconstruct the faces $\fb, \gb, \hb$ in a $2$-dimensional space (a
process indicated by the rightmost arrow), so that we see that we can indeed
hope to reconstruct $\Fb$ as long as we know how to obtain the joint invariants
$\jb_{a,b}$ from the invariants $\Jb_i$. This link between the invariants is
represented by the bottom arrow.

In the context of ternary quartic forms, the ``faces'' $\fb, \gb$ and $\hb$
mentioned above are the analogues of three binary forms $\bb_8, \bb_4$ and
$\bb_0$ of respective degree $8,4$ and $0$ (so that $\bb_0$ is a constant
polynomial in the invariants). The group $\SL_2 (\C)$ has a simultaneous action
on these forms. Olive~\cite{olive} has recently determined a corresponding set
of fundamental invariants $\jb_{d_1,d_2}$. Among these joint invariants are a
set of fundamental invariants $\jb_{d_1,0}$ of the octic form $\bb_8$, similar
fundamental invariants $\jb_{0,d_2}$ of the quartic $\bb_4$, and the form
$\bb_0$ itself, but the invariant algebra contains many more ``cross-terms''
beyond these.

The correspondence between $\Fb$ and the triple $\bbv = (\bb_8, \bb_4, \bb_0)$
is induced by a classical isomorphism between $\SL_2 (\C) / \{\pm 1\}$ and the
special orthogonal group $\SO (q)$ of a fixed (co- or contravariant) ternary
quadratic form $q$. Using Lie theory, one constructs an isomorphism between $\C
[R_{3,4}]^{\SO (q)}$ and the ring of invariants $\C [R_{2,8} \oplus R_{2,4}
\oplus R_{2,0}]^{\SL_2(\C)}$ for the diagonal action of $\SL_2 (\C)$. This idea
was exploited by Katsylo~\cite{katsylo-bi} and used in a more explicit way by
Van Rijnswou~\cite{VanRijnswou}. We refer to Section~\ref{sec:vr} for details
on this part of our argument.

We are interested in invariants under the whole group $\SL_3 (\C)$, not merely
those of $\SO (q)$. Fortunately, it is possible to reduce the study of the
former invariants to that of the latter by using the notion of $(G,
H)$-sections (see Definition~\ref{def:section}). Roughly speaking, one can show
that a generic quartic $F$ is $\SL_3 (\C)$-equivalent to one in the set $\Zs$
of quartics whose quadratic contravariant $\rhob (F)$ (as defined
in~\eqref{eq:rho}) is a non-zero multiple $u \cdot q$ of the chosen standard
quadratic form $q$. Moreover, up to scalar multiplication two generic quartics
in $\Zs$ are equivalent under the action of $\SL_3 (\C)$ if and only if they
are equivalent under the action of $\SO_3 (q)$.

As was known to Katsylo (but see also the first part of
Proposition~\ref{prop:section}) this means that the function field $\C
(R_{3,4})^{\SL_3 (\C)}$ is isomorphic to $\C (\Zs)^{\SO_3 (q)}$. We can then
use the correspondence between $F$ and the triple $\bv = (b_8, b_4, b_0)$.
After this, we show that one can actually control the denominator of this
expression, as well as the degree of the numerator. This is achieved by
combining a regularity statement (see the second part of
Proposition~\ref{prop:section}) with a detailed study of a transformation
matrix $T$ that transforms a generic quadric to one whose covariant is of the
form $u \cdot q$ (see Section~\ref{sec:normalization}) and a fundamental
relation between $\bb_0$, the Dixmier--Ohno invariants $\Ib_9, \Ib_{12}$ and
the determinant of $T$ (see Lemma~\ref{lem:I9b0}). 

Our final theoretic description of the numerator and denominator is given in
Theorem~\ref{thm:main}, for which we present two different proofs. Using
interpolation then allows us to recover explicit expressions for the joint
invariants in terms of the Dixmier--Ohno invariants, as is described further in
Section~\ref{sec:interpolation}.

The final step is to actually perform the reconstruction of a triple $(b_8,
b_4, b_0)$ once values of the joint invariants are given. We first show that if
$F$ is a \defi{stable} quartic (which is in particular the case for quartics
with non-zero discriminant) and $\Ib_{12} (F) \ne 0$, then the corresponding
triple $(b_8, b_4, b_0)$ is in the stable locus of the action of $\SL_2(\C)$
and therefore its orbit is uniquely determined by its joint invariants (see
Theorem~\ref{thm:rec_locus}). Reconstruction is therefore theoretically
possible under these hypotheses; it remains to show how to do this explicitly
in the generic situation.

Our reconstruction algorithm in Section~\ref{sec:rec_algs} starts from given
Dixmier--Ohno invariants defined over field $k$ of characteristic $0$. In order
to avoid working over a cubic extension of the field of definition of the
invariants, we normalize the invariant $b_0$ so it becomes equal to the
Dixmier--Ohno invariant $I_9$ and then we use the explicit relations of
Section~\ref{sec:interpolation} to get the corresponding joint invariants
$j_{d_1,d_2}$. We then use the methods from~\cite{LR11} for reconstructing a
stable binary octic $b_8$ from given Shioda invariants, which can be derived
from the $j_{i,0}$ (see Remark~\ref{rem:shioda}). Note that this step, which
uses the Mestre's method, will in general require a quadratic extension of the
base field $k$. One can then use the values $j_{d,1}$ of the joint invariants
$\jb_{d,1}$ which are linear in the coefficients of $b_4$, to determine the
coefficients of $b_4$ by solving a linear system. Finally, one  transforms back
the triple $(b_8,b_4,b_0)$ by a linear isomorphism $\ell^*$ (given
in~\eqref{eq:lstar_matrix}) to find a ternary quartic $F$ which Dixmier--Ohno
invariants are projectively equal to the given ones as a point in a weighted
projective space. We briefly indicate in Section~\ref{sec:descent} how one can
perform a Galois descent if one wishes to find an $F$ over $k$, which in fact
is generically possible. Finally, a form $F$ over $k$ whose Dixmier--Ohno
invariants are projectively equivalent to the given ones can generically also
be scaled over $k$ to obtain an exact equality of invariants.

The implementation of our results that was used when writing this article can
be found at~\cite{LRS16-Code}.
\medskip

A number of open problems remain. A first of these is to remove our genericity
assumptions; one would certainly like to be able to deal with the quartics in
the locus $\Ib_{12} = 0$ as well, and more generally, to be able to perform the
reconstruction for any tuple of Dixmier--Ohno invariants corresponding to a
stable quartic. Another open problem is the extension of our theory to fields
of positive characteristic, for which it is not even clear when the
Dixmier--Ohno invariants are still a fundamental set.

\subsection*{Acknowledgments}

We would like to thank the participants of the working group TEDI, and in
particular Boris Kolev and Marc Olive, for their interest and for the many
useful discussions.

\subsection*{Notation and conventions}

While we recall most conventions at the beginning of the relevant sections,
here are some general conventions to which we adhere in this article.

In what follows, $K$ is an algebraically closed base field, supposed to be of
characteristic $0$. Forms over $K$ are denoted with Roman letters, which are
capitalized when dealing with ternary forms. The universal forms of forms, as
well as those of their covariants, are denoted in boldface Roman letters, so
that we for example see individual ternary forms $F$ as incarnations of the
universal ternary quartic form $\Fb$. Conversely, when specializing at a given
form over $K$, this boldface is removed, so that the value of a covariant $\Cb$
at a form $F$ over $K$ is denoted by $C$. Contravariants and their values are
denoted by Greek letters, so that a contravariant $\gammab$ has a
specialization $\gamma = \gammab (F)$ at a form $F$ over $K$.

For the sake of clarity, we denote group actions with an intermediate dot, so
that we write $g . s$ where one would often merely find $g s$.

\section{Invariant theory}

\subsection{Formalism}
\label{sec:formalism}

In this section we describe the actions of linear groups that we will need in
the rest of the article, as well as the formalism needed to deal with
invariants, covariants, and contravariants. This exposition is adapted to our
needs and therefore not in any sense complete; other general discussions of the
theory of invariants and covariants can be found in~\cite[\S~1]{kraft-procesi}
and~\cite[Lecture~5]{dolgainv}.

Given a finite-dimensional vector space $V$ over an algebraically closed field
$K$, let $V^* = \Hom (V, K)$ be its dual. Since $K$ is infinite, homogeneous
polynomial functions on $V$ can be naturally identify with forms, which are
defined as follows.

\begin{definition}
  A \defi{form} over $K$ is an element $F$ of a vector space $\Sym^n (V^*)$,
  where $V$ is a finite-dimensional vector space over $K$.  Given a form $F$,
  its \defi{arity} (or more colloquially its \defi{number of variables}) is
  the dimension $m$ of $V$. The \defi{degree} of $F$ is denoted by $n$.
\end{definition}

Given $V$, we also define the \defi{polynomial ring} over $V$ to be the
graded ring
\begin{equation}
  K [V] = \bigoplus_{n \ge 0} \Sym^n (V^*) ,
\end{equation}
and we define the \defi{affine space} $\A_V$ over $V$ to be the variety $\Spec
K [V]$. In particular, the functor that to a vector space associates the
corresponding affine space is covariant.

In this article, we will consider group actions on several vector spaces and
duals. We will occasionally want to switch actions from left to right, and as
such, given a left (resp.\ right) action of a group $G$ on a vector space $V$,
we convert it into a right (resp.\ left) action by setting
\begin{equation}\label{eq:action_switch}
  v . T := T^{-1} . v \quad (\text{resp.} \; T . v := v . T^{-1})
\end{equation}
for $T \in G$ and $v \in V$.

A subgroup $G$ of $\GL (V)$ has a natural left action on $V$. This induces a
natural right action on $V^*$; for $x \in V^*$ and $T \in G$ we have $x . T = x
\circ T$. These actions induce further natural actions on $\Sym^n (V)$ and
$\Sym^n (V^*)$.

Choosing a basis $\vv = \left( v_1, \dots , v_m \right)$ of $V$, we get a dual
basis $\xv = \left( x_1 , \dots , x_m \right) = \left(v_1^*, \dots , v_m^*
\right)$ of $V^*$. Given $n$ and a list $I = (i_1, \dots, i_m)$ of $m$
non-negative integers the sum of which is $n$, we denote $x^I = x_1^{i_1}
\cdots x_m^{i_m} \in \Sym^n (V^*)$. The elements $x^I$ form a basis of
$\Sym^n(V^*)$, and we denote the elements of the corresponding dual basis
$\av$ by $a_I \in \Sym^n (V)$. For a form $F$, we have $(x^I)^* (F) = a_I
(F)$, so that we can write $F = \sum a_I (F) x^I$ and identify the values $a_I
(F) \in K$ with the coefficients of $F$ in this expression.

A choice of basis for $V$ also allows us to identify an element $T$ of $\GL
(V)$ with a matrix $[ T ] = \left( t_{i,j} \right)_{i,j = 1}^{m}$. With this
notation, the result $T . v_i$ of having $T$ act on the left on the $i$th
vector $v_i$ of the basis $\vv$ of $V$ corresponds to the $i$th column of the
matrix $[T]$. By duality, the result $x_i . T$ of having $T$ act on the right
on the $i$th vector $x_i$ of the dual basis $\xv$ of $V^*$ corresponds to the
$i$th row of the matrix $[T]$, or alternatively to the $i$th column of the
transpose of $[T]$. In a formula, we have
\begin{equation}\label{eq:action_subst}
  \begin{split}
    T.v_i &= t_{1,i} v_1 + \dots + t_{m,i} v_m
    \\ \textrm{(resp.} \; x_i.T &=  t_{i,1} x_1 + \dots + t_{i,m} x_m ) .
  \end{split}
\end{equation}
We will abbreviate
\begin{equation}
  T . \vv = (T . v_1, \dots, T . v_m)
  \quad (\text{resp.} \;
  \xv . T = (x_1 . T, \dots, x_m . T)\,) .
\end{equation}

\begin{remark}
  The substitution $\xv. T$ corresponding to the right action of $T$ on $\xv$
  (and that also induces the right action of $T$ on the symmetric powers
  $\Sym^n (V^*)$) is in fact encoded in the formal matrix product
  \begin{equation}\label{eq:action_subst_alt}
    \begin{pmatrix}
      t_{1,1} & \hdots & t_{1,n} \\
      \vdots & \ddots & \vdots \\
      t_{n,1} & \hdots & t_{n,n}
    \end{pmatrix}
    \times
    \begin{pmatrix}
      x_1 \\
      \vdots \\
      x_n
    \end{pmatrix}
    .
  \end{equation}
  Note how in~\eqref{eq:action_subst_alt} the \emph{right} action on the dual
  is obtained by using variables instead of values for the usual \emph{left}
  action of $[T]$ on column vectors; this is a general principle when
  dualizing. The substitution $T. \vv$ corresponding to the left action of $T$
  on $V$ can be obtained by multiplying the formal row vector corresponding to
  $\vv$ on the right by $T$ instead.
\end{remark}

The individual forms $F \in \Sym^n (V^*)$ can be seen as the incarnations of a
universal form, which we will denote by $\Fb$. In coordinates, this is the form
\begin{equation}\label{eq:Funiv}
  \Fb := \Fb (\av, \xv) := \Fb (\left( a_I \right), \left( x^I \right)) =
  \sum_{I : \sum I = n} a_I x^I \in \Sym^n (V) \otimes \Sym^n (V^*) .
\end{equation}
The duality  $\langle x^I, T. a_J \rangle = \langle T^* (x^I), a_J \rangle =
\langle x^I . T, a_J \rangle$ implies  the fundamental compatibility
\begin{equation}\label{eq:F_compat}
  \Fb (T . \av, \xv) := \Fb (\left(T . a_I \right), \left( x^I \right)) =
  \Fb (\left( a_I \right), \left( x^I . T\right)) =: \Fb (\av , \xv . T) .
\end{equation}
This also follows because $\Fb$ corresponds to the canonical bilinear
contraction (or evaluation) $\Sym^n (V^*) \times \Sym^n (V) \to K$. A form $F$
over $K$ can be obtained from the universal form $\Fb$ by contracting it with
the morphism $\Sym^n (V) \to K$ that corresponds to $F \in \Sym^n (V^*)$.

Using the coordinate vectors $\av$ and $\xv$ (resp.\ $\vv$) in what follows
will allow us to concretely represent covariant (resp.\ contravariant) tensors
as bihomogeneous polynomials of degree $d$ in the $a_I$ and degree $r$ in
$x^I$ (resp.\ $v^I$).

The following notion is a fundamental tool in the study of the representations
$\Sym^n (V^*)$ of $\GL (V)$ and its subgroups.

\begin{definition}\label{def:variant}
  A \defi{covariant} (resp.\ \defi{contravariant}) of $\Sym^n (V^*)$ is an $\SL
  (V)$-equivariant homogeneous polynomial map
  \begin{equation}\label{eq:variant}
    \begin{split}
      & \Cb : \Sym^n (V^*) \to \Sym^r (V^*) \\
      ( \text{resp.} \;
      & \gammab : \Sym^n (V^*) \to \Sym^r (V) ) .
    \end{split}
  \end{equation}
  The \defi{order} of $\Cb$ (resp.\ $\gammab$) is defined to be $r$, whereas
  the \defi{degree} of $\Cb$ (resp.\ $\gammab$) is its degree (in $\av$) as a
  homogeneous polynomial map.

  An \defi{invariant} is a covariant (or, for that matter, a contravariant) of
  order $0$.
\end{definition}

\begin{remark}
  In a sense, we could have decided only to work with covariants, since after
  agreeing that $\Sym^{r} (V^*) = \Sym^{-r} (V)$ for $r$ negative, a
  contravariant becomes nothing but a covariant of negative order. This
  description would lead to a fully unified description of co- and
  contravariants and a more concise formalism. We have however rather chosen to
  follow the classical description, in which co- and contravariants are
  distinguished.

  At any rate, we have already restricted ourselves to a special case. In
  general (see~\cite[\S~1.4]{kraft-procesi} and
  especially~\cite[\S~5.2]{dolgainv}), given a representation $V$ of a group
  $G$, a \defi{$W$-covariant} of $V$ is a $G$-equivariant homogeneous
  polynomial map $V \to W$, where $W$ is another representation of $G$. As was
  classically the custom, we have restricted ourselves to covariants of the
  same arity as the original form. Note that in order to study joint covariants
  (which is in fact needed to determine the joint invariants mentioned at the
  end of this section) we would have to use the more general notion just
  mentioned.
\end{remark}

We proceed to unwind our definitions. Consider a co- or contravariant as in
\eqref{eq:variant}, of order $r$ and degree $d$.
\begin{enumerate}
\item We can consider the homogeneous polynomial map~\eqref{eq:variant} as an
  equivariant linear map $\Sym^d (\Sym^n (V^*)) \to \Sym^r (V^*)$ (resp.\
  $\Sym^d (\Sym^n (V^*)) \to \Sym^r (V)$), which in turn is nothing but an
  invariant tensor
  \begin{equation}\label{eq:variant_symsym}
    \begin{split}
      & \Cb \in (\Sym^d (\Sym^n (V)) \otimes \Sym^r (V^*))^{\SL (V)} \\
      ( \text{resp.} \;
      & \gammab \in (\Sym^d (\Sym^n (V)) \otimes \Sym^r (V))^{\SL (V)} ) .
    \end{split}
  \end{equation}
\item In this form, the demand that $\Cb$ (resp.\ $\gammab$) be invariant
  translates into the analogue of~\eqref{eq:F_compat} (but only for
  transformations by $\SL (V)$; see~\eqref{eq:weight_co}
  and~\eqref{eq:weight_contra} below for the transformation behavior under
  $\GL (V)$). Indeed, for $T \in \SL (V)$ the invariance of $\Cb = \Cb (\av,
  \xv)$ under the diagonal action under $T$ translates to
  \begin{equation}
    \begin{split}
      & \Cb (\av, \xv) . T
      = \Cb (\av . T, \xv . T)
      = \Cb (T^{-1} . \av, \xv . T) \\
      ( \text{resp.} \;
      & \gammab (\av, \xv) . T
      = \gammab (\av . T, \xv . T)
      = \gammab (T^{-1} . \av, \xv . T) ) . \\
    \end{split}
  \end{equation}
  Applying $T$ on the factor $\Sym^n (V)$ then leads to
  \begin{equation}\label{eq:variant_compat}
    \begin{split}
      & \Cb (T . \av, \xv) = \Cb (\av , \xv . T) \\
      ( \text{resp.} \;
      &  \gammab (T . \av , \vv) =\gammab (\av, \vv . T)  ) .
    \end{split}
  \end{equation}
\item Putting all degrees and orders together, we can also consider $\Cb$
  (resp. $\gammab$) as a homogeneous element of bidegree $(d, r)$ of the
  algebra
  \begin{equation}
    \begin{split}
      & (K [ \Sym^n (V^*) ] \otimes K [ V ])^{\SL (V)}
      =
      K [ \Sym^n (V^*) \oplus V ]^{\SL (V)} \\
      ( \text{resp.} \;
      & (K [ \Sym^n (V^*) ] \otimes K [ V^* ])^{\SL (V)}
      =
      K [ \Sym^n (V^*) \oplus V ^*]^{\SL (V)} ).
    \end{split}
  \end{equation}
\item Or yet differently, by dualizing (i) we see that $\Cb$ (resp.\
  $\gammab$) is nothing but an inclusion of the irreducible representation
  $\Sym^r (V)$ (resp.\ $\Sym^r (V^*)$) of $\SL (V)$ into either the
  finite-dimensional representation $\Sym^d (\Sym^n (V))$ or the
  infinite-dimensional representation $K [ \Sym^n (V^*) ]$.
\end{enumerate}

By studying the action of scalar (or diagonal) matrices, we see that the action
of elements of $\GL (V)$ is also intertwined, but only up to a scalar. More
precisely, for covariants we have
\begin{equation}\label{eq:weight_co}
  \Cb (F . T) = \det (T)^{(n d - r) / m} \Cb (F) . T\,,
\end{equation}
while for contravariants we have
\begin{equation}\label{eq:weight_contra}
  \gammab (F . T) = \det (T)^{(n d + r) / m} \gammab (F) . T\,.
\end{equation}
The relevant (integral!) quotient $(n d - r) / m$ (resp. $(n d + r) / m$\,) is
called the \defi{weight} of the covariant (resp. contravariant).

\begin{remark}
  The point of view (iv) above that considers covariants and contravariants as
  subrepresentations is quite valuable when we want to make the theory
  explicit. Indeed, one can decompose tensor powers $\Sym^d (\Sym^n (V))$ in
  computer algebra packages such as \Magma~\cite{Magma} or, much faster, in
  \LiE~\cite{LiE}. Using these packages for ternary quartics (so $\dim (V) = 3$
  and $n = 4$) shows that the first non-trivial contravariant of order $2$
  occurs for $d = 4$, in which case there is a unique such covariant up to
  scaling. This is the contravariant $\rhob$ constructed in
  Section~\ref{sec:dixmier}, which has also been mentioned in passing in the
  introduction. 

  The first covariants of order $2$ show up in degree $d = 5$. In this case the
  factor $\Sym^2 (V)$ occurs with multiplicity $2$, so any covariant of degree
  $5$ can uniquely be written as a combination
  \begin{equation}
    \lambda_1 \Cb_1 + \lambda_2 \Cb_2 : \Sym^4 (V^*) \to \Sym^2 (V^*)
  \end{equation}
  for two fixed independent such covariants $\Cb_1$ and $\Cb_2$ that can be
  taken to be the covariants $\Tb$ and $\Xb$ from~\eqref{eq:tau}.
\end{remark}

In Section~\ref{sec:jointinvs}, we will also consider \defi{joint invariants}.
These arise when considering the action of a subgroup $G$ of $\GL (V)$ on a
reducible representation, in our case direct sums of vector spaces
\begin{equation}
  S = \Sym^{n_1} (V^*) \oplus \ldots \oplus \Sym^{n_t} (V^*)\,.
\end{equation}
We have $G$ act diagonally on $S$, so that
\begin{equation}
  (F_1, \ldots, F_t) . T = (F_1 . T, \ldots, F_t . T)
\end{equation}
for $T \in G$ and $(F_1, \ldots, F_t) \in S$. This gives rise to the invariant
algebra $K [S]^G$, which has a natural grading by multi-degree in the
coefficients of the forms $F_i$. A form that is pure of multi-degree $(d_1,
\ldots, d_t)$ is denoted by $\jb_{d_1, \dots , d_t}$. We will call $d = d_1 +
\ldots + d_t$ the \defi{total degree} of such an invariant.

\subsection{Dixmier--Ohno invariants}
\label{sec:dixmier}

In this section we briefly recall the construction and notation of some of the
invariants, covariants and contravariants of ternary quartics. We refer
to~\cite{dixmier,elsenhans,giko} for more details.

Let $V$ be a vector space of dimension $3$ over $K$ with basis $v_1, v_2, v_3$
and corresponding dual basis $x_1, x_2, x_3$. Let
\begin{equation}
  D : K [x_1, x_2, x_3] \times K [v_1, v_2, v_3] \to K [x_1, x_2, x_3]
\end{equation}
be the differential operator that extends the linear contraction pairing $(v_i,
x_j) \to \delta_{ij}$ by associating to a monomial $v_1^{i_1} v_2^{i_2}
v_3^{i_3}$ of degree $m$ the operator
\begin{equation}
  \frac{\partial^m}{\partial x_1^{i_1} \partial x_2^{i_2} \partial x_3^{i_3}}.
\end{equation}

If $Q (x_1,x_2,x_3)$ is a ternary form, we let
\begin{equation}
  H(Q) = \frac{1}{2}
  \begin{bmatrix}
    \frac{\partial^2 Q}{\partial x_1^2} & \frac{\partial^2 Q}{\partial x_1
      \partial x_2 } &\frac{\partial^2 Q}{\partial x_1 \partial x_3} \\
    \frac{\partial^2 Q}{\partial x_1 \partial x_2} &\frac{\partial^2
      Q}{\partial x_2^2} &\frac{\partial^2 Q}{\partial x_2 \partial x_3 } \\
    \frac{\partial^2 Q}{\partial x_1 \partial x_3} &\frac{\partial^2
      Q}{\partial x_2 \partial x_3} &\frac{\partial^2 Q}{\partial x_3^2}
  \end{bmatrix}
\end{equation}
be its Hessian matrix. For a quadratic form $Q$ over $K$, we let $H(Q)^*$ be
the classical adjoint matrix of its Hessian $H (Q)$. Given two square matrices
$A = (a_{ij}),B = (b_{ij})$ of the same dimension, we also denote $\langle A,B
\rangle = \sum_{i,j} a_{i j} b_{i j}$.

With this notation, given a ternary quadric covariant $\Qb(x_1,x_2,x_3)$ and a
ternary quadric contravariant $\rhob(v_1,v_2,v_3)$, we let
\begin{equation}
  J_{11} (\Qb,\rhob) = \langle H (\Qb),   H (\rhob)\rangle,
  \quad
  J_{22} (\Qb,\rhob) = \langle H (\Qb)^*, H (\rhob)^* \rangle
\end{equation}
and
\begin{equation}
  J_{30} (\Qb) = \det (H (\Qb)), \quad J_{03} (\rhob) = \det (H (\rhob)).
\end{equation}

Dixmier~\cite{dixmier} and Ohno~\cite{ohno} (but see also~\cite{elsenhans})
have used these operators to determine new covariants, contravariants and
invariants starting from the ternary quartic covariant $\Fb$, its Hessian $\Hb
= 216^{-1} \, \det(H(\Fb))$ and two contravariants $\sigmab$ (of degree $2$ and
order $4$) and $\psib$ (of degree $3$ and order $6$) that appear
in~\cite[\S~92,\S~292]{salmon}. First, they define a quadratic contravariant of
degree 4
\begin{equation} \label{eq:rho}
  \rhob = 144^{-1} D (\Fb,\psib).
\end{equation}
From $\rhob$, one can derive two quadratic covariants of degree $5$ (denoted
$\tau$ and $\xi$ in the papers cited above)
\begin{equation}\label{eq:tau}
  \Tb = 144^{-1} D (\Fb,\rhob),
  \quad
  \Xb = 72^{-1} D (\Hb, \sigmab)
\end{equation}
and then two other quadratic contravariants, of degree 7 and 13,
\begin{equation}
  \etab = 12^{-1} D (\Xb, \rhob),\quad \chib = 8^{-1} D( D(\Tb, \psib), \psib)
\end{equation}
as well as a linear covariant of degree 14 (which is denoted by $\nu$ in the
papers cite above), namely
\begin{equation}
  \Nb = 8^{-1} D (D (\Hb, \rhob), \etab)\,.
\end{equation}

The \defi{Dixmier invariants} are then the algebraically independent invariants
\begin{equation}\label{eq:dixmier}
  \begin{array}{lll}
    \Ib_3 = 144^{-1} D (\Fb, \sigmab),
    &
    \Ib_9 = J_{11} (\Tb,\rhob),
    &
    \Ib_{15} = J_{30} (\Tb),
    \\
    \Ib_6 = 4608^{-1} (D (\Hb, \psib) - 8 \Ib_3^2),
    &
    \Ib_{12} = J_{03} (\rhob),
    &
    \Ib_{18} = J_{22} (\Tb, \rhob)
    \\
    \multicolumn{3}{l}{\hspace*{1cm}\text{and }\Ib_{27} = \disc (\Fb)
      \text{\ \ (the discriminant of $\Fb$)}\,.}
    
  \end{array}
\end{equation}
To obtain the full ring of invariants of ternary quartics, we have
to add the \defi{Ohno invariants}
\begin{equation}\label{eq:ohno}
  \begin{array}{lll}
    \Jb_9 = J_{11} (\Xb, \rhob),
    &
    \Jb_{15} = J_{30} (\Xb),
    &
    \Ib_{21} = J_{03} (\etab),
    \\
    \Jb_{12} = J_{11} (\Tb, \etab),
    &
    \Jb_{18}= J_{22} (\Xb, \rhob),
    &
    \Jb_{21} = J_{11} (\Nb, \etab).
  \end{array}
\end{equation}

\begin{theorem}[Ohno]\label{thm:DO}
  The ring of invariants $K [\Sym^4 (V^*)]^{\SL (V)}$ is generated by the 13
  invariants~\eqref{eq:dixmier} and~\eqref{eq:ohno}.
\end{theorem}

The subscripts of the Dixmier--Ohno invariants indicate their degree in the
coefficients of the universal form $\Fb$. In what follows, this is the degree
to which we refer whenever we mention the degree of a homogeneous expression
in the Dixmier--Ohno invariants.

\begin{remark}
  The above theorem is originally due to Ohno~\cite{ohno}. Work by
  Elsenhans~\cite{elsenhans} gives a readable verification of his results.
  These invariants were further studied and implemented by Girard and Kohel
  in~\cite{giko}.
\end{remark}

\subsection{Joint invariants of octics and quartics}
\label{sec:jointinvs}

Let $W$ be a vector space of dimension $2$ over $K$ with basis $w_1, w_2$ and
corresponding dual basis $z_1, z_2$. We will now describe the basis of
fundamental invariants calculated in~\cite{olive} for the (diagonal) action of
$\SL (W)$ on $\Sym^8 (W^*) \oplus \Sym^4 (W^*)$.

The \defi{Clebsch--Gordan decomposition}~\cite{Fulton-Harris} is the
decomposition of $\SL (W)$-represen\-tations
\begin{equation}
  \Sym^{n_1} (W^*) \otimes \Sym^{n_2} (W^*)
  \simeq
  \bigoplus_{i=0}^{\min(n_1, n_2)} \Sym^{n_1 + n_2 -2i}(W^*).
\end{equation}
Let $i$ be an integer between $0$ and $\min (n_1, n_2)$. Then we can consider the
\defi{Clebsch-Gordan projector} (or \defi{transvectant})
\begin{equation}
  \Sym^{n_1} (W^*) \otimes \Sym^{n_2} (W^*)
  \longrightarrow
  \Sym^{n_1+n_2-2i}(W^*), \quad b_{n_1} \otimes b_{n_2}
  \mapsto 
  (b_{n_1}, b_{n_2})_{i}\,.
\end{equation}

An explicit formula for the transvectant operator~\cite{olver} can be obtained
by using the bi-differential \defi{Cayley operator},
\begin{equation}
  \Omega_{yz}(\,b_{n_1}(y_1, y_2)\cdot b_{n_2}(z_1,z_2)\,)
  :=
  \frac{\partial b_{n_1}}{y_1} \frac{\partial b_{n_2}}{z_2}
  -
  \frac{\partial b_{n_1}}{y_2} \frac{\partial b_{n_2}}{z_1}\,,
\end{equation}
and the \defi{polarization operator}
\begin{equation}
  \sigma_{y}(\,b_n(y_1,y _2)\,) := y_1 \frac{\partial b_n}{\partial y_1} + y_2
  \frac{\partial b_n}{\partial y_2}.
\end{equation}
Let $i$ be as before. Then the \defi{transvectant} of index $i$ is given by
\begin{equation}
  (b_{n_1}, b_{n_2})_i
  :=
  \frac{({n_1}-i)\,!}{{n_1}\,!}\, \frac{({n_2}-i)\,!}{{n_2}\,!}\,
  \Omega_{yz}^i\, \sigma_{y}^{{n_1}-i}\, \sigma_{z}^{{n_2}-i}\,
  \left(\,{b_{n_1}}(y_1,y_2)\cdot {b_{n_2}}(z_1,z_2)\,\right)\,.
\end{equation}

Now an important result of classical invariant theory states that taking the
closure of the universal pair $(\bb_{n_1},\,\bb_{n_2})\in \Sym^{n_1} (W^*)
\times \Sym^{n_2} (W^*)$ under appropriate transvectant operations generates
its covariant algebra~\cite{Pro1998}. Moreover, Gordan's algorithm enables one
to construct a minimal basis of this algebra~\cite{gordan68}.

Using this algorithm in the case $n_1 = 8$ and $n_2 = 4$, Olive constructed
Table~\ref{tab:jinv}, where the 63 elements of the invariant basis have been
overlaid in gray. We denote them by $\jb_{d_8,d_4}$, sorting by bidegree as
mentioned at the end of Section~\ref{sec:formalism}.

Two invariants in Olive's basis are pure invariants of $\bb_4$ (\textit{cf.}
Table~\ref{tab:covs4}), 9 of them are pure invariants of $\bb_8$ (\textit{cf.}
Table~\ref{tab:covs8}) and the 52 remaining are joint invariants (\textit{cf.}
Table~\ref{tab:s4s8}). From this table, we can for instance deduce that
\begin{multline*}
  2^{9} \cdot 3^{2} \cdot 5^{2} \cdot 7^{2} \times \jb_{1,2} =%
  140\,{{ b^{2}_{4,4}}}\,{ b_{8,0}}-35\,{ b_{4,3}}\,{ b_{4,4}}\,{ b_{8,1}}+\\
  \left( 10\,{ b_{4,2}}\,{ b_{4,4}}+5\,{{ b^{2}_{4,3}}
    } \right) { b_{8,2}}
  -\left( 5\,{ b_{4,1}}\,{ b_{4,4}}+5\,{ b_{4,2}}\,{ b_{4,3}} \right) {
    b_{8,3}}+\\
  \left( 4\,{
      b_{4,0}}\,{ b_{4,4}}+4\,{ b_{4,1}}\,{ b_{4,3}}+2\,{{ b^{2}_{4,2}}} \right) { b_{8,4}}- \left( 5\,{ b_{4,0}}\,{ b_{4,3}}+
    5\,{ b_{4,1}}\,{ b_{4,2}} \right) { b_{8,5}}\\
  + \left( 10\,{ b_{4,0}}\,{ b_{4,2}}+5\,{{ b^{2}_{4,1}}} \right) { b_{8,6}} - 35\,{
    b_{4,0}}\,{ b_{4,1}}\,{ b_{8,7}}+140\,{{ b^{2}_{4,0}}}\,{ b_{8,8}}\,.
\end{multline*}

\begin{table}[htbp]
  \caption{A minimal basis of 63 invariants for $\Sym^{8}(W^*)\oplus\Sym^{4}(W^*)$\,.}%
  \label{tab:jinv}
  \resizebox{0.95\linewidth}{!}{
    \begin{minipage}[b]{1.0\linewidth}
      \begin{center}
        \begin{footnotesize}
          \begin{math}
            \setlength{\arraycolsep}{0.5mm}
            \begin{array}{%
                rl>{\scriptstyle }c|%
                c@{\,\,}%
                rlc@{\,}|%
                rl>{\scriptstyle }c|%
                c@{\,\,}rlc@{\,}|%
                rl>{\scriptstyle }c|%
              }
              \multicolumn{2}{c}{Covariant} & _{(d,o)} &&
              \multicolumn{3}{c|}{Invariant} &
              \multicolumn{2}{c}{Covariant} & _{(d,o)} &
              \multicolumn{4}{c|}{Invariant} &
              \multicolumn{2}{c}{Covariant} & _{(d,o)} \\
              \hline
              \hline
              \gb_{1} & =\,\bb_4 & _{(1,4)} &%
              &\ctbl \jb_{0, 2} & \ctbl =\,\trv{\gb_{1}}{\gb_1}{4} &&
              \gb_{2} & =\,\trv{\gb_{1}}{\gb_1}{2} & _{(2,4)}  &%
              &\ctbl \jb_{0, 3} & \ctbl =\,\trv{\gb_{2}}{\gb_1}{4} &&
              \gb_{3} & =\,\trv{\gb_{2}}{\gb_1}{2} & _{(3,6)}\\
            \end{array}
          \end{math}
        \end{footnotesize}
      \end{center}\medskip
      \subcaption{Covariants of $\Sym^{4}(W^*)$.}%
      \label{tab:covs4}\medskip
    \end{minipage}}
  \hspace*{-1.2cm}\resizebox{0.95\linewidth}{!}{
    \begin{minipage}[b]{1.0\linewidth}
      \begin{center}
        \begin{footnotesize}
          \begin{math}
            \setlength{\arraycolsep}{0.4mm}
            \begin{array}{c@{\,}rlc@{\,}c|c@{\,\,}rlc@{\,}c|c@{\,\,}rlc@{\,}c|c@{\,\,}rlc@{\,}c}
              \multicolumn{4}{c}{Covariant} & _{(d,o)} &
              \multicolumn{4}{c}{Covariant} & _{(d,o)} &
              \multicolumn{4}{c}{Covariant} & _{(d,o)} &
              \multicolumn{4}{c}{Covariant} & _{(d,o)} \\
              \hline
              \hline
              &\fb_{1} & =\,\bb_8 && _{(1,8)} &%
              &\fb_{19} & =\,\trv{\fb_{13}}{\fb_1}{8} && _{(4,10)}  &%
              &\fb_{37} & =\,\trv{\fb_{33}}{\fb_1}{7} && _{(6,4)}  &%
              &\fb_{55} & =\,\trv{\fb_{51}}{\fb_1}{5} && _{(8,4)}\\
              &\ctbl{}\jb_{2,0} & \ctbl{}=\,\trv{\fb_1}{\fb_1}{8} && &%
              &\fb_{20} & =\,\trv{\fb_{12}}{\fb_1}{6} && _{(4,10)}  &%
              &\fb_{38} & =\,\trv{\fb_{32}}{\fb_1}{7} && _{(6,4)}  &%
              &\fb_{56} & =\,\trv{\fb_{50}}{\fb_1}{5} && _{(8,4)}\\
              &\fb_{3} & =\,\trv{\fb_1}{\fb_1}{6} && _{(2,4)}  &%
              &\fb_{21} & =\,\trv{\fb_{13}}{\fb_1}{7} && _{(4,12)}  &%
              &\fb_{39} & =\,\trv{\fb_{34}}{\fb_1}{8} && _{(6,6)}  &%
              &\fb_{57} & =\,\trv{\fb_{51}}{\fb_1}{4} && _{(8,6)}\\
              &\fb_{4} & =\,\trv{\fb_1}{\fb_1}{4} && _{(2,8)}  &%
              &\fb_{22} & =\,\trv{\fb_{13}}{\fb_1}{6} && _{(4,14)}  &%
              &\fb_{40} & =\,\trv{\fb_{33}}{\fb_1}{6} && _{(6,6)}  &%
              &\fb_{58} & =\,\trv{\fb_{50}}{\fb_1}{4} && _{(8,6)}\\
              &\fb_{5} & =\,\trv{\fb_1}{\fb_1}{2} && _{(2,12)}  &%
              &\fb_{23} & =\,\trv{\fb_{13}}{\fb_1}{4} && _{(4,18)}  &%
              &\fb_{41} & =\,\trv{\fb_{32}}{\fb_1}{6} && _{(6,6)}  &%
              &\ctbl{}\jb_{9,0} & \ctbl{}=\,\trv{\fb_{15}\fb_{16}}{\fb_1}{8} &&\\
              &\ctbl{}\jb_{3,0} & \ctbl{}=\,\trv{\fb_{4}}{\fb_1}{8} && &%
              &\ctbl{}\jb_{5,0} & \ctbl{}=\,\trv{\fb_{3}^{2}}{\fb_1}{8} && &%
              &\fb_{42} & =\,\trv{\fb_{34}}{\fb_1}{7} && _{(6,8)}  &%
              &\fb_{60} & =\,\trv{\fb_{58}}{\fb_1}{6} &&_{(9,2)}\\
              &\fb_{7} & =\,\trv{\fb_{5}}{\fb_1}{8} && _{(3,4)}  &%
              &\fb_{25} & =\,\trv{\fb_{20}}{\fb_1}{8} && _{(5,2)}  &%
              &\fb_{43} & =\,\trv{\fb_{34}}{\fb_1}{6} && _{(6,10)}  &%
              &\fb_{61} & =\,\trv{\fb_{57}}{\fb_1}{6} && _{(9,2)}\\
              &\fb_{8} & =\,\trv{\fb_{5}}{\fb_1}{7} && _{(3,6)}  &%
              &\fb_{26} & =\,\trv{\fb_{21}}{\fb_1}{8} && _{(5,4)}  &%
              &\ctbl{}\jb_{7,0} & \ctbl{}=\,\trv{\fb_{7}^{2}}{\fb_1}{8} && &%
              &\fb_{62} & =\,\trv{\fb_{16}\fb_{17}}{\fb_1}{8} &&_{(9,2)} \\
              &\fb_{9} & =\,\trv{\fb_{5}}{\fb_1}{6} && _{(3,8)}  &%
              &\fb_{27} & =\,\trv{\fb_{20}}{\fb_1}{7} && _{(5,4)}  &%
              &\fb_{45} & =\,\trv{\fb_{43}}{\fb_1}{8} && _{(7,2)} &%
              &\fb_{63} & =\,\trv{\fb_{58}}{\fb_1}{5} && _{(9,4)}\\
              &\fb_{10} & =\,\trv{\fb_{5}}{\fb_1}{5} && _{(3,10)}  &%
              &\fb_{28} & =\,\trv{\fb_{22}}{\fb_1}{8} && _{(5,6)}  &%
              &\fb_{46} & =\,\trv{\fb_{42}}{\fb_1}{7} && _{(7,2)} &%
              &\ctbl{}\jb_{10,0} & \ctbl{}=\,\trv{\fb_{17}\fb_{25}}{\fb_1}{8}&&  \\
              &\fb_{11} & =\,\trv{\fb_{5}}{\fb_1}{4} && _{(3,12)}  &%
              &\fb_{29} & =\,\trv{\fb_{21}}{\fb_1}{7} && _{(5,6)}  &%
              &\fb_{47} & =\,\trv{\fb_{43}}{\fb_1}{7} && _{(7,4)} &%
              &\fb_{65} & =\,\trv{\fb_{17}\fb_{27}}{\fb_1}{8} &&_{(10,2)}\\
              &\fb_{12} & =\,\trv{\fb_{5}}{\fb_1}{3} && _{(3,14)}  &%
              &\fb_{30} & =\,\trv{\fb_{22}}{\fb_1}{7} && _{(5,8)}  &%
              &\fb_{48} & =\,\trv{\fb_{42}}{\fb_1}{6} && _{(7,4)} &%
              &\fb_{66} & =\,\trv{\fb_{17}\fb_{26}}{\fb_1}{8} &&_{(10,2)}\\
              &\fb_{13} & =\,\trv{\fb_{5}}{\fb_1}{1} && _{(3,18)}  &%
              &\fb_{31} & =\,\trv{\fb_{23}}{\fb_1}{8} && _{(5,10)}  &%
              &\fb_{49} & =\,\trv{\fb_{43}}{\fb_1}{6} && _{(7,6)} &%
              &\fb_{67} & =\,\trv{\fb_{27}\fb_{29}}{\fb_1}{8} &&_{(11,2)}\\
              &\ctbl{}\jb_{4,0} & \ctbl{}=\,\trv{\fb_{9}}{\fb_1}{8} && &%
              &\fb_{32} & =\,\trv{\fb_{22}}{\fb_1}{6} && _{(5,10)}  &%
              &\fb_{50} & =\,\trv{\fb_{42}}{\fb_1}{5} && _{(7,6)} &%
              &\fb_{68} & =\,\trv{\fb_{27}\fb_{28}}{\fb_1}{8} &&_{(11,2)}\\
              &\fb_{15} & =\,\trv{\fb_{11}}{\fb_1}{8} && _{(4,4)}  &%
              &\fb_{33} & =\,\trv{\fb_{21}}{\fb_1}{5} && _{(5,10)}  &%
              &\fb_{51} & =\,\trv{\fb_{41}}{\fb_1}{4} && _{(7,6)} &%
              &\fb_{69} & =\,\trv{\fb_{29}\fb_{38}}{\fb_1}{8}&& _{(12,2)}\\
              &\fb_{16} & =\,\trv{\fb_{10}}{\fb_1}{7} && _{(4,4)}  &%
              &\fb_{34} & =\,\trv{\fb_{23}}{\fb_1}{6} && _{(5,14)}  &%
              &\ctbl{}\jb_{8,0} & \ctbl{}=\,\trv{\fb_{7}\fb_{16}}{\fb_1}{8} && \\
              &\fb_{17} & =\,\trv{\fb_{12}}{\fb_1}{8} && _{(4,6)}  &%
              &\ctbl{}\jb_{6,0} & \ctbl{}=\,\trv{\fb_{3}\fb_{7}}{\fb_1}{8} &&  &%
              &\fb_{53} & =\,\trv{\fb_{51}}{\fb_1}{6} && _{(8,2)}\\
              &\fb_{18} & =\,\trv{\fb_{12}}{\fb_1}{7} && _{(4,8)}  &%
              &\fb_{36} & =\,\trv{\fb_{33}}{\fb_1}{8} && _{(6,2)}  &%
              &\fb_{54} & =\,\trv{\fb_{50}}{\fb_1}{6} && _{(8,2)}\\
            \end{array}
          \end{math}
        \end{footnotesize}
      \end{center}\medskip
      \subcaption{Covariants of $\Sym^{8}(W^*)$.}%
      \label{tab:covs8}\medskip
    \end{minipage}}
  \resizebox{0.95\linewidth}{!}{
    \begin{minipage}[b]{1.0\linewidth}
      \begin{center}
        \begin{footnotesize}
          \begin{math}
            \setlength{\arraycolsep}{0.7mm}
            \begin{array}{c@{\,}>{\ctbl}r>{\ctbl}lc@{\,}|c@{\,\,}>{\ctbl}r>{\ctbl}lc@{\,}|c@{\,\,}>{\ctbl}r>{\ctbl}lc@{\,}|c@{\,\,}>{\ctbl}r>{\ctbl}lc@{\,}}
              \multicolumn{4}{c|}{Invariant} &
              \multicolumn{4}{c|}{Invariant} &
              \multicolumn{4}{c|}{Invariant} &
              \multicolumn{4}{c}{Invariant} \\
              \hline
              \hline
              &\jb_{1, 2}&=\,\trv{\fb_1}{\gb_1^2}{8}&&%
              &\jb_{2, 4}&=\,\trv{\fb_4}{\gb_{2}^2}{8}&&%
              &\jb_{4, 3}&=\,\trv{\fb_{18}}{\gb_1\gb_{2}}{8}&&%
              &\jb_{5, 3}&=\,\trv{\fb_{29}}{\gb_{3}}{6}&\\
              &\jb_{2, 1}&=\,\trv{\fb_3}{\gb_1}{4}&&%
              &\jb'_{2, 4}&=\,\trv{\fb_5}{\gb_1^2\gb_{2}}{12}&&%
              &\jb'_{4, 3}&=\,\trv{\fb_{17}}{\gb_{3}}{6}&&%
              &\jb'_{5, 3}&=\,\trv{\fb_{30}}{\gb_1\gb_{2}}{8}&\\
              &\jb_{1, 3}&=\,\trv{\fb_1}{\gb_1\gb_{2}}{8}&&%
              &\jb_{3, 3}&=\,\trv{\fb_{11}}{\gb_1^3}{12}&&%
              &\jb''_{4, 3}&=\,\trv{\fb_{21}}{\gb_1^3}{12}&&%
              &\jb_{6, 2}&=\,\trv{\fb_{37}}{\gb_{2}}{4}&\\%
              &\jb_{2, 2}&=\,\trv{\fb_4}{\gb_1^2}{8}&&%
              &\jb'_{3, 3}&=\,\trv{\fb_9}{\gb_1\gb_{2}}{8}&&%
              &\jb_{5, 2}&=\,\trv{\fb_{30}}{\gb_1^2}{8}&&%
              &\jb'_{6, 2}&=\,\trv{\fb_{38}}{\gb_{2}}{4}&\\%
              &\jb'_{2, 2}&=\,\trv{\fb_3}{\gb_{2}}{4}&&%
              &\jb''_{3, 3}&=\,\trv{\fb_8}{\gb_{3}}{6}&&%
              &\jb'_{5, 2}&=\,\trv{\fb_{27}}{\gb_{2}}{4}&&%
              &\jb''_{6, 2}&=\,\trv{\fb_{42}}{\gb_1^2}{8}&\\%
              &\jb_{3, 1}&=\,\trv{\fb_7}{\gb_1}{4}&&%
              &\jb_{4, 2}&=\,\trv{\fb_{15}}{\gb_{2}}{4}&&%
              &\jb''_{5, 2}&=\,\trv{\fb_{26}}{\gb_{2}}{4}&&%
              &\jb_{6, 3}&=\,\trv{\fb_8^2}{\gb_1^3}{12}&\\
              &\jb_{1, 4}&=\,\trv{\fb_1}{\gb_{2}^2}{8}&&%
              &\jb'_{4, 2}&=\,\trv{\fb_{18}}{\gb_1^2}{8}&&%
              &\jb_{6, 1}&=\,\trv{\fb_{37}}{\gb_1}{4}&&%
              &\jb_{7, 2}&=\,\trv{\fb_{48}}{\gb_{2}}{4}&\\
              &\jb_{2, 3}&=\,\trv{\fb_4}{\gb_1\gb_{2}}{8}&&%
              &\jb''_{4, 2}&=\,\trv{\fb_{16}}{\gb_{2}}{4}&&%
              &\jb'_{6, 1}&=\,\trv{\fb_{38}}{\gb_1}{4}&&%
              &\jb'_{7, 2}&=\,\trv{\fb_{47}}{\gb_{2}}{4}&\\
              &\jb'_{2, 3}&=\,\trv{\fb_5}{\gb_1^3}{12}&&%
              &\jb_{5, 1}&=\,\trv{\fb_{26}}{\gb_1}{4}&&%
              &\jb_{7, 1}&=\,\trv{\fb_{47}}{\gb_1}{4}&&%
              &\jb_{8, 1}&=\,\trv{\fb_{55}}{\gb_1}{4}&\\
              &\jb_{3, 2}&=\,\trv{\fb_7}{\gb_{2}}{4}&&%
              &\jb'_{5, 1}&=\,\trv{\fb_{27}}{\gb_1}{4}&&%
              &\jb'_{7, 1}&=\,\trv{\fb_{48}}{\gb_1}{4}&&%
              &\jb'_{8, 1}&=\,\trv{\fb_{56}}{\gb_1}{4}&\\
              &\jb'_{3, 2}&=\,\trv{\fb_9}{\gb_1^2}{8}&&%
              &\jb_{2, 5}&=\,\trv{\fb_5}{\gb_1\gb_{2}^2}{12}&&%
              &\jb_{3, 5}&=\,\trv{\fb_{11}}{\gb_1\gb_{2}^2}{12}&&%
              &\jb_{8, 2}&=\,\trv{\fb_{56}}{\gb_{2}}{4}&\\
              &\jb_{4, 1}&=\,\trv{\fb_{15}}{\gb_1}{4}&&%
              &\jb_{3, 4}&=\,\trv{\fb_{10}}{\gb_1\gb_{3}}{10}&&%
              &\jb_{4, 4}&=\,\trv{\fb_{20}}{\gb_1\gb_{3}}{10}&&%
              &\jb_{9, 1}&=\,\trv{\fb_{63}}{\gb_1}{4}&\\
              &\jb'_{4, 1}&=\,\trv{\fb_{16}}{\gb_1}{4}&&%
              &\jb'_{3, 4}&=\,\trv{\fb_{11}}{\gb_1^2\gb_{2}}{12}&&%
              &\jb'_{4, 4}&=\,\trv{\fb_{21}}{\gb_1^2\gb_{2}}{12}&&%
              &\jb_{10, 1}&=\,\trv{\fb_{25}^2}{\gb_1}{4}&\\
            \end{array}
          \end{math}
        \end{footnotesize}
      \end{center}\medskip
      \subcaption{Joint invariants}%
      \label{tab:s4s8}
    \end{minipage}}
\end{table}

\begin{remark}\label{rem:shioda}
  Note that the invariant basis $\{\jb_{2,0}, \jb_{3,0}, \ldots \jb_{10,0}\}$
  of $\Sym^{8}(W^*)$ given in Table~\ref{tab:covs8} differs slightly from the
  basis $\{\Jb_2, \Jb_3, \ldots, \Jb_{10}\}$ of~\cite{shioda67}. We have
  \begin{multline}
    \Jb_2 = 2^{7}\cdot3^{2}\cdot5\cdot7\ \jb_{2,0}\,,\ %
    \Jb_3 = 2^{10}\cdot3^{3}\cdot5\cdot7\ \jb_{3,0}\,,\ \\%
    \Jb_4 = 2^{13}\cdot3^{3}\cdot5^{2}\,(\,50\ \jb_{2,0}^2-33\ \jb_{4,0}\,) %
  \end{multline}
  and similar equations exist for the other $\Jb_d$.
\end{remark}

\section{Sections and rationality}

\subsection{A result by Van Rijnswou}
\label{sec:vr}

In this section we recall an explicit isomorphism of $\SL (W)$-representations
that was first constructed in~\cite{VanRijnswou}. For reasons that we mentioned
in the introduction it has proved to be an extremely fruitful inroad into our
problem.

Let $W$ be a vector space of dimension $2$ over $K$ with basis $w_1, w_2$ and
corresponding dual basis $z_1, z_2$. Consider the vector space $V = \Sym^2
(W)$, which is of dimension $3$. We fix the basis $v_1 = w_1^2$, $v_2 = 2 w_1
w_2$, $v_3 = w_2^2$ of $V$, with corresponding dual basis $x_1 = z_1^2$, $x_2
= \frac{1}{2} z_1 z_2$, $x_3 = z_2^2$. The induced left action of $\SL (W)$ on
$V$ in the bases chosen above is described by the homomorphism
\begin{align}\label{eq:GL2ToGL3}
  \begin{split}
    h : \SL (W) & \to \SL (\Sym^2 (W)) = \SL (V) \\
    \begin{pmatrix}
      a & b \\
      c & d
    \end{pmatrix}
    & \mapsto
    \begin{pmatrix}
      a^2 &   2 a b   & b^2 \\
      a c & a d + b c & b d \\
      c^2 &   2 c d   & d^2
    \end{pmatrix}
  \end{split}
\end{align}
since for example $w_1^2$ is sent to $(a w_1 + c w_2)^2 = a^2 w_1^2 + a c (2
w_1 w_2) + c^2 w_2^2 = a^2\, v_1 + a c\, v_2 + c^2\, v_3$.

It is known that the substitution action of $\SL (W)$ on the space of binary
forms $\Sym^2 (W^*)$ is orthonormal with respect to the discriminant form. So
too is the action on $\Sym^2 (W)$ above orthogonal with respect to the dual of
this discriminant form, which we calculate to be
\begin{equation}
  (z_1 z_2)^2 - 4 z_1^2 z_2^2 = 4 (x_2^2 - x_1 x_3)\,.
\end{equation}
In other words, the homomorphism~\eqref{eq:GL2ToGL3} has image in $\SO (x_2^2
- x_1 x_3)$. Because its tangent map is surjective, it is itself surjective,
as $\SO (x_2^2 - x_1 x_3)$ is irreducible. Though we will not use this fact,
we mention that its kernel has order $2$.

Similarly the homomorphism
\begin{align}\label{eq:GL2ToGL3Dual}
  \begin{split}
    h^* : \SL (W) & \to \SL (\Sym^2 (W)) = \SL (V) \\
    \begin{pmatrix}
      a & b \\
      c & d
    \end{pmatrix}
    & \mapsto
    \begin{pmatrix}
      a^2   &    a b    & b^2   \\
      2 a c & a d + b c & 2 b d \\
      c^2   &    c d    & d^2
    \end{pmatrix}
  \end{split}
\end{align}
is a degree $2$ surjection onto $\SO (v_2^2 - v_1 v_3)$. Note that the notation
$h^*$ is slightly abusive here, because it is not $h$ itself that is dualized;
rather, as a subgroup of $\SL (V)$ the image of $h$ fixes the element $x_2^2 -
x_1 x_3$ while $h^*$ fixes its dual $v_2^2 - v_1 v_3$.

The irreducible representations of $\SL (W)$ are up to isomorphism described by
the symmetric powers $\Sym^d (W^*)$ by~\cite[Proposition~3.2.1]{springer}. By
\cite[Theorem 5.4]{dolgainv}, there exists a decomposition
\begin{equation}\label{eq:equivariant_dumb}
  \Sym^4 (\Sym^2 (W^*))
  \cong
  \Sym^8 (W^*) \oplus \Sym^4 (W^*) \oplus \Sym^0 (W^*)
\end{equation}
of $\SL (W)$-representations. Following~\cite{VanRijnswou}, we will now
indicate how to calculate an explicit isomorphism~\eqref{eq:equivariant_dumb}.
We focus on the covariant case to fix ideas, so that we use the homomorphism
$h$; in the contravariant case one works with $h^*$ instead.

Consider the classical generators
\begin{equation}\label{eq:sl_gens}
  \el =
  \begin{pmatrix}
    0 & 1 \\
    0 & 0
  \end{pmatrix}, \;
  \hl =
  \begin{pmatrix}
    1 & 0 \\
    0 & -1
  \end{pmatrix}, \;
  \fl =
  \begin{pmatrix}
    0 & 0 \\
    1 & 0
  \end{pmatrix}
\end{equation}
of the Lie algebra $\sll (W)$ of $\SL (W)$, which acts on the right on both
sides of~\eqref{eq:equivariant_dumb} by taking the differential of the action
of the latter group. For example, the one-parameter subgroup of $\SL (W)$ given
by
\begin{equation}\label{eq:E_prim2}
  \begin{pmatrix}
    1 & \eps \\
    0 & 1
  \end{pmatrix}
\end{equation}
gives rise to $\el$ upon deriving with respect to the parameter $\eps$ at
$\eps = 0$.
The action of the matrix~\eqref{eq:E_prim2} sends $(z_1, z_2)$ to $(z_1 +
\eps z_2, z_2)$. Deriving the latter expression with respect to $\eps$ at $\eps
= 0$ we obtain $(z_2, 0)$, which indeed equals $(z_1, z_2) . \el$\,.

The right action on $\Sym^d (W^*)$ is derived from~\eqref{eq:sl_gens} by the
Leibniz rule; $\el$ sends a binary form $b$ to $
z_1 . \el\ \frac{\partial b}{\partial z_1} +  z_2 . \el \ \frac{\partial
  b}{\partial z_2} = z_2 \frac{\partial
  b}{\partial z_1}$, and similarly for the other generators.

Taking the derivative at the identity of the map $h$ in~\eqref{eq:GL2ToGL3}, we
obtain the action on $V$, which is represented by the matrices
\begin{equation}\label{eq:ad}
  (D h) \el =
  \begin{pmatrix}
    0 & 2 & 0 \\
    0 & 0 & 1 \\
    0 & 0 & 0
  \end{pmatrix}, \;
  (D h) \hl =
  \begin{pmatrix}
    2 & 0 & 0 \\
    0 & 0 & 0 \\
    0 & 0 & -2
  \end{pmatrix}, \;
  (D h) \fl =
  \begin{pmatrix}
    0 & 0 & 0 \\
    1 & 0 & 0 \\
    0 & 2 & 0
  \end{pmatrix}
\end{equation}
To verify the equalities in~\eqref{eq:ad}, we can also employ a more direct
argument. For example, via $h$ the action of the matrices~\eqref{eq:E_prim2}
transforms into that of
\begin{equation}\label{eq:E_prim3}
  \begin{pmatrix}
    1 & 2 \eps & \eps^2 \\
    0 & 1 & \eps \\
    0 & 0 & 1
  \end{pmatrix}
\end{equation}
and indeed upon deriving with respect to $\eps$ at $\eps = 0$ we obtain the
first matrix in~\eqref{eq:ad}.

In what follows, we consider $V$ as an $\SL (W)$-representation
via~\eqref{eq:ad}, so that the left action of $\el$ (resp.\ $\hl$, $\fl$) is
represented by the matrix $(D h) \el$ (resp.\ $(D h) \hl$, $(D h) \fl$)
in~\eqref{eq:ad} instead. Using the Leibniz rule we see can once more derive
the right action on $\Sym^4 (V^*)$: for example, the action of $\el$ sends a
ternary quartic form $F$ to
\begin{equation}
  2 x_2 \frac{\partial F}{\partial x_1}
  + x_3 \frac{\partial F}{\partial x_2} .
\end{equation}

We now use the theory of highest weight vectors on these representations of
$\sll (W)$. In a nutshell, the special case of this theory for representations
of $\sll (W)$ (concerning which a fuller account is available
in~\cite{Fulton-Harris}) amounts to the following observation. For the
standard representations $\Sym^p (W^*)$ of $\SL (W)$ we observe that there is
up to scaling a unique vector $u$ whose weight for the action of $\hl$ is the
highest possible, namely $u = z_1^p$.  This vector also has the property that
$u . \fl = 0$, and moreover the complete representation is spanned by the $p +
1$ powers $\left\{ u . \el^{(0)}, u .  \el^{(1)}, \ldots , u . \el^{(p)}
\right\}$, which indeed are up to scaling the vectors $\left\{ z_1^p,
  z_1^{p-1} z_2, \dots , z_2^p \right\}$.

To decompose an arbitrary finite-dimensional representation $U$ of $\sll (W)$
into its irreducible constituents, we first track down a basis of the subspace
whose elements are of highest weight $\lambda_0$ under the action of $\hl$. We
know by the above that their images under $\el$ generate a subrepresentation
that is isomorphic to a power $\Sym^{p_0} (W^*)$, the multiplicity being given
by the dimension of the subspace. Subsequently, we determine the subspace of
all vectors $u \in U$ that are of the next highest weight $\lambda_1$ under
$\hl$ but not yet in the image of the subrepresentations found and that
additionally satisfy $u . \fl = 0$ (a condition satisfied automatically for the
vectors of highest weight $\lambda_0$). The space of these vectors gives rise
to another subrepresentation with multiplicity as above. Continuing in this
way, we obtain our decomposition.

In our case we obtain that a subrepresentation of the left hand side
of~\eqref{eq:equivariant_dumb} that is isomorphic to $\Sym^8 (W^*)$ is
generated by the successive images under $\el$ the vector of highest weight
$x_1^4$, which is unique up to a scalar multiple. We map this vector to
$z_1^8$, the vector of highest weight in $\Sym^8 (W^*)$, which then forces our
hand by equivariance.

Another subrepresentation, isomorphic to $\Sym^4 (W^*)$, can be found by
considering a vector of weight $4$ that is annihilated by $\fl$. This is $x_1^3
x_3 - x_1^2 x_2^2$ or one of its multiples. We map this vector to $z_1^4$, the
vector of highest weight in $\Sym^4 (W^*)$, and again our hand is forced.
Finally, a new vector of weight $0$ that is annihilated by $\fl$ is given by
$x_1^2 x_3^2 - 2 x_1 x_2^2 x_3 + x_2^4$ or one of its multiples. We map this to
$1 \in \Sym^0 (W^*)$, and we finally have an equivariant
isomorphism~\eqref{eq:equivariant_dumb} for the homomorphism $h$. After
proceeding similarly for the dual version $h^*$ from~\eqref{eq:GL2ToGL3Dual},
we have two equivariant linear maps
\begin{equation}\label{eq:equivariant}
  \begin{array}{rcl}
    \ell, \ell^* : \Sym^4 (V^*) & \to &
    \Sym^8 (W^*) \oplus \Sym^4 (W^*) \oplus \Sym^0 (W^*)\,,\\[0.2cm]
    F & \mapsto & (b_8, b_4, b_0) ,
  \end{array}
\end{equation}
and our construction shows that the following theorem holds.

\begin{theorem}[Van Rijnswou, 2001~\cite{VanRijnswou}]
  We have
  \begin{equation}\label{eq:l_equiv}
    \ell (F . h(T)) = \ell(F) . T
  \end{equation}
  and
  \begin{equation}\label{eq:lstar_equiv}
    \ell^* (F . h^* (T)) = \ell^* (F) . T\,.
  \end{equation}
  In other words, $\ell$, (resp.\ $\ell^*$) gives an $\SL (W)$-equivariant
  linear map
  \begin{equation}
    \Sym^4 (\Sym^2 (W^*)) \to \Sym^8 (W^*) \oplus \Sym^4 (W^*) \oplus \Sym^0 (W*)
  \end{equation}
  when $\Sym^4 (\Sym^2 (W^*))$ is considered as an $\SL (W)$-representation via
  $h$ (resp.\ $h^*$).
\end{theorem}
Note that the maps $\ell$ and $\ell^*$ are not unique, but in light of Schur's
Lemma (or the explicit construction above) they are up to a scalar multiple per
irreducible component.

In coordinates, we can describe these maps as follows. Let a basis of the
left-hand side of~\eqref{eq:equivariant} be given by
\begin{equation}
  \begin{scriptstyle}
    x_1^4,\,
    x_1^3 x_2,\,
    x_1^3 x_3,\,
    x_1^2 x_2^2,\,
    x_1^2 x_2 x_3,\,
    x_1^2 x_3^2,\,
    x_1 x_2^3,\,
    x_1 x_2^2 x_3,\,
    x_1 x_2 x_3^2,\,
    x_1 x_3^3,\,
    x_2^4,\,
    x_2^3 x_3,\,
    x_2^2 x_3^2,\,
    x_2 x_3^3,\,
    x_3^4
  \end{scriptstyle}
\end{equation}
and let a basis for the right-hand side be given by
\begin{equation}
  \begin{scriptstyle}
    z_1^8,\,
    z_1^7 z_2,\,
    z_1^6 z_2^2,\,
    z_1^5 z_2^3,\,
    z_1^4 z_2^4,\,
    z_1^3 z_2^5,\,
    z_1^2 z_2^6,\,
    z_1 z_2^7,\,
    z_2^8,\,\ \ \ 
    z_1^4,\,
    z_1^3 z_2,\,
    z_1^2 z_2^2,\,
    z_1 z_2^3,\,
    z_2^4,\,\ \ \ 
    1
  \end{scriptstyle}
\end{equation}

Then with respect to these bases, $\ell$ is given by the right action of
\begin{equation}\label{eq:l_matrix}
  \renewcommand*{\arraystretch}{0.6}
  \left(
    \begin{array}{C{0.5em}C{0.5em}C{0.5em}C{0.5em}C{0.5em}C{0.5em}C{0.5em}C{0.5em}C{0.5em}C{0.5em}C{0.5em}C{0.5em}C{0.5em}C{0.5em}C{0.5em}}
      1 & . & . & . & . & . & . & . & . & . & . & . & . & . & . \\
      . & 1 & . & . & . & . & . & . & . & . & . & . & . & . & . \\
      . & . & 1 & . & . & . & . & . & . & \frac{6}{7} & . & . & . & . & . \\
      . & . & 1 & . & . & . & . & . & . & \text{-}\frac{1}{7} & . & . & . & . & . \\
      . & . & . & 1 & . & . & . & . & . & . & \frac{4}{7} & . & . & . & . \\
      . & . & . & . & 1 & . & . & . & . & . & . & \frac{8}{7} & . & . & \frac{8}{15} \\
      . & . & . & 1 & . & . & . & . & . & . & \text{-}\frac{3}{7} & . & . & . & . \\
      . & . & . & . & 1 & . & . & . & . & . & . & \frac{1}{7} & . & . & \text{-}\frac{2}{15} \\
      . & . & . & . & . & 1 & . & . & . & . & . & . & \frac{4}{7} & . & . \\
      . & . & . & . & . & . & 1 & . & . & . & . & . & . & \frac{6}{7} & . \\
      . & . & . & . & 1 & . & . & . & . & . & . & \text{-}\frac{6}{7} & . & . & \frac{1}{5} \\
      . & . & . & . & . & 1 & . & . & . & . & . & . & \text{-}\frac{3}{7} & . & . \\
      . & . & . & . & . & . & 1 & . & . & . & . & . & . & \text{-}\frac{1}{7} & . \\
      . & . & . & . & . & . & . & 1 & . & . & . & . & . & . & . \\
      . & . & . & . & . & . & . & . & 1 & . & . & . & . & . & . \\
    \end{array}
  \right)
\end{equation}
whereas $\ell^*$ is given by the right action of
\begin{equation}\label{eq:lstar_matrix}
  \renewcommand*{\arraystretch}{0.6}
  \left(
    \begin{array}{C{0.5em}C{0.5em}C{0.5em}C{0.5em}C{0.5em}C{0.5em}C{0.5em}C{0.5em}C{0.5em}C{0.5em}C{0.5em}C{0.5em}C{0.5em}C{0.5em}C{0.5em}}
      1 & . & . & . & . & . & . & . & . & . & . & . & . & . & . \\
      . & 2 & . & . & . & . & . & . & . & . & . & . & . & . & . \\
      . & . & 1 & . & . & . & . & . & . & \text{-}\frac{3}{14} & . & . & . & . & . \\
      . & . & 4 & . & . & . & . & . & . & \frac{1}{7} & . & . & . & . & . \\
      . & . & . & 2 & . & . & . & . & . & . & \text{-}\frac{2}{7} & . & . & . & . \\
      . & . & . & . & 1 & . & . & . & . & . & . & \text{-}\frac{2}{7} & . & . & \frac{1}{30} \\
      . & . & . & 8 & . & . & . & . & . & . & \frac{6}{7} & . & . & . & . \\
      . & . & . & . & 4 & . & . & . & . & . & . & \text{-}\frac{1}{7} & . & . & \text{-}\frac{1}{30} \\
      . & . & . & . & . & 2 & . & . & . & . & . & . & \text{-}\frac{2}{7} & . & . \\
      . & . & . & . & . & . & 1 & . & . & . & . & . & . & \text{-}\frac{3}{14} & . \\
      . & . & . & . & 16 & . & . & . & . & . & . & \frac{24}{7} & . & . & \frac{1}{5} \\
      . & . & . & . & . & 8 & . & . & . & . & . & . & \frac{6}{7} & . & . \\
      . & . & . & . & . & . & 4 & . & . & . & . & . & . & \frac{1}{7} & . \\
      . & . & . & . & . & . & . & 2 & . & . & . & . & . & . & . \\
      . & . & . & . & . & . & . & . & 1 & . & . & . & . & . & . \\
    \end{array}
  \right)
\end{equation}
where we have replaced zeros with dots for clarity.

\begin{remark}
  One can explicitly verify that the linear maps $\ell$, $\ell^*$ defined by
  the matrices~\eqref{eq:l_matrix} and~\eqref{eq:lstar_matrix} still define
  equivariant maps in the sense of~\eqref{eq:l_equiv}
  and~\eqref{eq:lstar_equiv} over fields of characteristic larger than $7$.
\end{remark}

\subsection{$(G,H)$-sections}\label{sec:GHsec}

We will now follow part of a beautiful procedure investigated by Katsylo
in~\cite{katsylo} (see also~\cite{bohning}) to prove the rationality of the
moduli space of genus $3$ curves. It is based on the following notion.

\begin{definition}\label{def:section}
  Let $G$ be a linear algebraic group acting on an irreducible quasi-projective
  variety $\Ys$. Let $f : H \to G$ be a group morphism, through which we
  consider $H$ to act on $\Ys$. Let $\Zs$ be a closed subvariety of $\Ys$. Then
  $\Zs$ is called a \defi{$(G,H)$-section} of the action of $G$ on $\Ys$ if
  \begin{enumerate}
  \item the stabilizer of $\Zs$ in $\Ys$ is the image of $f$;
  \item there exists an open subset $\Zs_1$ of $\Zs$ such that any two points
    of $\Zs_1$ that are $G$-equivalent in $\Ys$ are in fact $H$-equivalent in
    $\Zs$;
  \item $\Ys$ is the closure of $G \cdot \Zs$.
  \end{enumerate}
\end{definition}

\begin{remark}
  While the notion of a $(G, H)$-section depends on the morphism from $H$ to
  $G$, we do not mention this morphism explicitly in what follows.
\end{remark}

\begin{proposition}[{\cite[\S~3]{gatti-viniberghi}}]\label{prop:section}
  Suppose that $\Zs$ is a $(G,H)$-section of $\Ys$. Then the canonical
  restriction arrow $K (\Ys)^G \to K (\Zs)^H$ between fields of rational
  functions is an isomorphism. If we additionally assume that
  \begin{enumerate}
  \item $\Ys$ is an affine normal variety,
  \item $G$ is a linear algebraic group that does not admit any non-trivial
    character, and
  \item any closed orbit of $G$ in $\Ys$ intersects $\Zs$,
  \end{enumerate}
  then the canonical restriction arrow $K [\Ys]^G \to K [\Zs]^H$ between rings
  of regular functions is an isomorphism.
\end{proposition}
\begin{proof}
  This proposition is proved in~\cite[Proposition~4]{gatti-viniberghi} for the
  affine case. The generalization of the first part of the result follows from
  Rosenlicht's theorem~\cite[Theorem~6.2]{dolgainv}; this shows the existence
  of an invariant affine open in $\Zs$ whose field of $G$-invariant rational
  functions equals that of $\Zs$. The proof
  in~\cite[Proposition~4]{gatti-viniberghi} then still applies.
\end{proof}

The situation in which we make use of Proposition~\ref{prop:section} is the
following. We let $\Ys$ be the variety whose $K$-points are given by
\begin{equation}
  \begin{array}{lll}
    \Ys (K)
    &=& \left\{ F \in \Sym^4 (V^*) : \disc (\rhob (F)) \neq 0 \right\}\,.\\
  \end{array}
\end{equation}
Note that we may choose any co- or contravariant of order 2 to define $\Ys
(K)$, but since among these $\rhob$ is of smallest degree (namely $4$), it is
natural to focus on it here.
Consequently,
\begin{equation}
  \begin{array}{lll}
    \Ys (K)
    &=& \left\{ F \in \Sym^4 (V^*) : \Ib_{12} (F) \neq 0 \right\}\,.
  \end{array}
\end{equation}
That is, $\Ys$ is the open subvariety of $\Sym^4 (V^*)$ whose geometric points
correspond to those ternary quartics $F$ for which the contravariant $\rhob
(F)$ does not degenerate. The subvariety $\Zs$ is given by the locus of
quartics on which the contravariant $\rhob$ is normalized, so that
\begin{equation}
  \Zs (K)
  = \left\{ F \in \Sym^4 (V^*) : \text{$\rhob (F) = u\,(v_2^2-v_1v_3)$ for some
      $u \in K^*$} \right\} .
\end{equation}

We let $G = \SL (V)$. Note that by our restriction to the open affine where
$\Ib_{12} \neq 0$ we have that the ring of invariants
\begin{equation}
  K [\Ys]^G = K [ \Sym^4 (V^*) ]^G_{\Ib_{12}}
\end{equation}
is the localization of the ring of Dixmier--Ohno invariants with respect to
$\Ib_{12}$.

The stabilizer of $\Zs$ in $G$ is given by $\langle \zeta_3 \rangle \cdot \SO
(v_2^2 - v_1 v_3)$ where the cube root of unity $\zeta_3$ acts via its
inclusion into $G$ as a multiple of the identity matrix. This stabilizer
$\langle \zeta_3 \rangle \cdot \SO (v_2^2 - v_1 v_3)$ is the image of $H =
\langle \zeta_3 \rangle \cdot \SL (W)$ (where again $\zeta_3$ is considered as a
multiple of the identity) by the homomorphism
\begin{equation}
  f = h^* :
  H = \langle \zeta_3 \rangle \SL (W)
  \to
  \SL (V) = G
\end{equation}
given in~\eqref{eq:GL2ToGL3Dual}. Note that $\Zs$ is defined by $5$ equations
of degree $4$, and that the regular function $\ub$ on $\Zs$ defined by $\rhob
(F) = \ub (F) (v_2^2 - v_1 v_3)$ is invariant under $\SO (v_2^2 - v_1 v_3)$. It
transforms by $\zeta_3 . \ub = \zeta_3 \ub$ under the action of $\langle
\zeta_3 \rangle$. Putting all of this together, we obtain the following result.

\begin{lemma}\label{lem:sec}
  With definitions and notations as above, % we have that
  $\Zs$ is a $(G, H)$-section of $\Ys$, and % that
  the canonical restriction map $K [\Ys]^G \to K [\Zs]^H$ is an
  isomorphism. Moreover, we have that $K [\Zs]^{\SL (W)} = K [\Zs]^{H} [\ub]$.
\end{lemma}
\begin{proof}
  We have already mentioned that hypothesis (i) of
  Definition~\ref{def:section} is satisfied. We can take $\Zs_1 = \Zs$ in (ii)
  since by covariance any isomorphism between two forms in $\Zs$ fixes the
  form $v_2^2 - v_1 v_3$ up to a scalar. Since $\SL (V)$ acts transitively on
  non-degenerate quadric curves, we then also have that $\Ys = G . \Zs$ so
  that certainly (iii) holds.  To show the second claim, we have to check that
  the assumptions in Proposition~\ref{prop:section} hold in our case. Because
  $\Ys$ is an open subset of an affine space, we have (i). Moreover $G = \SL
  (V)$ does not admit any non-trivial character, so that (ii) is also
  verified. Again (iii) follows from our considerations in
  Section~\ref{sec:vr}; in fact any orbit of $G$ in $\Ys$ intersects $\Zs$.

  To obtain the final statement of the lemma, it remains to investigate the
  action of the central factor $\langle \zeta_3 \rangle$ of $H$ on the ring $K
  [\Zs]^{\SL (W)}$. We can decompose
  \begin{equation} \label{eq:decompo}
    K [\Zs]^{\SL (W)} = \bigoplus_{i=0}^2 M_i
  \end{equation}
  where $\langle \zeta_3 \rangle$ acts on the $i$-th $K$-submodule $M_i$ as
  scalar multiplication by $\zeta_3^i$. We have $M_0 = K [\Zs]^{H}$, so that
  our claim follows from the fact that $M_i = \ub^i M_0$ for $i > 0$. This, in
  turn, is a consequence of $\ub$ being a unit in $K [\Zs]$; indeed, $ \rhob
  (F) = \ub (F) \cdot (v_2^2 - v_1 v_3)$, and taking the discriminant on both
  sides, we get $\disc \rhob (F) = \Ib_{12} = \ub (F)^3 \cdot \disc (v_2^2 -
  v_1 v_3)$ (see also~\eqref{eq:disc}).
\end{proof}

Now let $\Ys' = \Sym^8 (W^*) \oplus \Sym^4 (W^*) \oplus \Sym^0 (W^*)$. In
Section~\ref{sec:vr} we have seen that there is an isomorphism of $\SL
(W)$-representations $\ell^* : \Ys \to \Ys'$. Let $\Zs'$ be the image of $\Zs$
under the isomorphism $\ell^*$. Then $\ell^*$ induces an isomorphism $K
[\Zs]^{\SL (W)} \to K [\Zs']^{\SL (W)}$.

Let $\jb$ be the restriction to $\Zs'$ of a degree $d$ joint invariant of $\SL
(W) \subset H$ on $\Ys'$. Then certainly $(\jb \ell^*)^3$ is an $H$-invariant
function on $\Zs$. Therefore by Proposition~\ref{prop:section} we can write
\begin{equation}
  (\jb \ell^*)^3 = \frac{\Qb}{\Ib_{12}^n}
\end{equation}
for some $n \in \Z$ and for some polynomial $\Qb$ in the Dixmier--Ohno
invariants that we may suppose not to be divisible by $\Ib_{12}$. Note that
$\Qb$ is homogeneous because $\jb$ and $\Ib_{12}$ are.

We see that $\jb \ell^* \in K [Z]^{\SL (W)}$ acquires a character under the
action of $\langle \zeta_3 \rangle$. Now recall from the end of the proof of
Lemma~\ref{lem:sec} that we have $\ub^3 = \Ib_{12} / 4$.  Therefore by our
hypothesis $\Ib_{12} \nmid \Qb$ and~\eqref{eq:decompo} we see that $\Qb$ is
actually in $K [\Zs]^H$, and we obtain the following.

\begin{proposition}\label{prop:rationality}
  Let $\jb$ be a joint invariant of degree $d$ for the action of $\SL (W)$ on
  $\Sym^8 (W^*) \oplus \Sym^4 (W^*) \oplus \Sym^0 (W^*)$. Then the composition
  $\jb \ell^*$, considered as a regular function on $\Zs \subset \Ys$, admits
  an expression of homogeneous degree $d$ of the form
  \begin{equation}
    \jb \ell^* = \frac{\Pb}{\ub^n}
  \end{equation}
  where $n \ge 0$ and where $\Pb$ is a polynomial in the Dixmier--Ohno invariants
  of homogeneous degree $d + 4 n$. Moreover, $(\jb \ell^*)^3$ extends to $\Ys$,
  where it is described by an expression of homogeneous degree $3 d$ of the
  form
  \begin{equation}
    (\jb \ell^*)^3 = \frac{\Pb^3}{\Ib_{12}^n}\,.
  \end{equation}
\end{proposition}

We need to analyze the pole of $\jb$ at the zero locus of $\ub$. In order to do
this, we will construct an ``association'' $\Ys \to \Zs$ that, given a plane
quartic $F$ corresponding to a point of $\Ys$, returns an isomorphic quartic in
$\Zs$. This association, while not a morphism of varieties, is still defined by
algebraic means in terms of universal forms. Moreover, it does descend to give
a well-defined morphism of quotient spaces. Its construction is the objective
of the next section, which will also show, among other things, that for the
$\SL (W)$-invariant function $\bb_0$ that is homogeneous of degree $1$ we have
that $\Ib_9 = 5 \ub^2 \cdot (\bb_0 \ell^*)$ (\cf\ \infra,
Lemma~\ref{lem:I9b0}).

\subsection{Normalizing quadrics}
\label{sec:normalization}

In order to prove our results and ultimately to perform our computations, we
need to transform a generic contravariant quadric into a multiple of $v_2^2 -
v_1 v_3$. To do this, we forget about contravariants for a moment and start out
with the universal ternary quadric form
\begin{equation}\label{eq:univ_quadric}
  \Qb = a x_2^2 + b x_2 x_1 + c x_2 x_3 + d x_1^2+ e x_1 x_3 + f x_3^2
\end{equation}
over $K$. We want to find a universal linear transformation that transforms
$\Qb$ to a form $\tilde{\Qb} = u (x_2^2 - x_1 x_3)$ with $u \in K [a, b, c, d,
e, f]$. Concretely, this will be a linear transformation over a finite
extension of $K (a, b, c, d, e,f )$ that sends $\Qb$ to a non-zero multiple of
$x_2^2 - x_1 x_3$.

We will construct this change of variables in such a way that:
\begin{enumerate}
\item we only take one square root;
\item the coefficients of the transformations are homogeneous of degree zero in
  the coefficients of $\Qb$;
\item $\ub$ is a monomial of small degree (in fact we will have $\ub = a$).
\end{enumerate}
As our transformations will involve a square root, we need to define the degree
of an expression involving such roots in order to make sense of (ii) above.

\begin{definition}\label{def:algdegree}
  Consider a ring $R = K [a_1, \ldots, a_n, r]$, where $a_i$ are algebraically
  independent variables and where $r$ is a root of an irreducible polynomial
  $x^m +c_1 x^{m-1}+ \ldots + c_m$ such the $c_i$ are homogeneous elements
  of $K (a_1, \ldots, a_n)$ whose degree is linear in $i$. We define the
  homogeneous degrees of the $a_i$ to be $1$, and that of $r$ to be $\deg(c_m)
  / m$. The homogeneous degree of an element $r$ of $R$, is defined to be the
  common homogeneous degree of the monomials in a representation of $r$, if
  such a common degree exists. 
\end{definition}
Note that this definition is independent of the choice of representation of
$r$ because of the properties of the minimal polynomial of $r$. We obtain a
decomposition $R = \oplus_d R_d$ by homogeneous degree.
We extend this grading to the fraction field of $R$ in the usual way. In
particular, a quotient of the form $P\,/\,Q$ of homogeneous polynomials $P$
and $Q$ of degree \,$\deg(P)$, respectively \,$\deg(Q)$, is said to have
homogeneous degree \,$\deg(P) - \deg(Q)$.\medskip

Let $\delta = a d - b^2/4$, $\eta = a e - b c/2$, and
\begin{equation}\label{eq:disc}
  \Deltab = a d f  - a e^2/4 - b^2 f/4 + b c e/4 - c^2 d/4
\end{equation}
so that $\Deltab$ is the discriminant of $\Qb$. To obtain our normalization,
we proceed as follows:
\begin{enumerate}
\item eliminate the monomials $x_2 x_1$ and $x_2 x_3$ of $\Qb$ by applying
  the transformation
  \begin{equation}
    x_2 \leftarrow x_2 - \frac{b}{2a} x_1 -\frac{c}{2 a} x_3 .
  \end{equation}
  We call the result $Q_1$.
\item eliminate the monomial $x_1 x_3$ in $Q_1$ by applying
  \begin{equation}
    x_1 \leftarrow x_1 - \frac{\eta}{2 \delta} x_3 .
  \end{equation}
  We call the result $Q_2$.
\item Let $r$ be such that $r^2 + a\, \Deltab = 0$ (so that $r$ has homogeneous
  degree $2$). Transform $Q_2$ into $Q_3 = a x_2^2 + (\delta/a) x_1 x_3$ by
  applying
  \begin{equation}
    x_1 \leftarrow \frac{x_1 + x_3}{2},
    \quad
    x_3 \leftarrow \frac{\delta}{2 r} \left( x_1 - x_3 \right).
  \end{equation}
\item Transform $Q_3$ into normalized form by applying
  \begin{equation}
    x_1 \leftarrow - \frac{\Deltab}{a^2} x_1 .
  \end{equation}
\end{enumerate}

Composing all these maps, we obtain that $T_0 \,.\, \Qb$ is a multiple of
$x_2^2 - x_1 x_3$, where
\begin{equation}\label{eq:T0}
  T_0 =
  \begin{pmatrix}
    \displaystyle-\frac{\delta}{a^2} &  0 & \displaystyle-\frac{2 r+\eta}{2 a^2} \\[0.3cm]
    \displaystyle \frac{b}{2a}  & 1  & \displaystyle\frac{c}{2a} \\[0.3cm]
    1 & 0 & \displaystyle-\frac{r-\eta}{\delta}
  \end{pmatrix}.
\end{equation}
When applying this to a contravariant, as we will indeed do later, we instead
have to transform by applying $T_0^{-1}$ on the right.

For later use, we also define the integral matrix
\begin{equation}\label{eq:Tint}
  T_{\intt} = a^3 \delta^2 T_0\,.
\end{equation}
One computes
\begin{equation}\label{eq:elem}
  \begin{aligned}
    \det(T_{\intt})^2
    & = & (a^9 \delta^6)^2 \det (T_0)^2\,, \\
    & = & (a^9 \delta^6)^2 \cdot {2^2 r^2}/{a^4}\,, \\
    & = & - 2^2 a^{15} \delta^{12} \Deltab\,.
  \end{aligned}
\end{equation}

Note how the matrices constructed in this section depend on the choice of a
square root. Because of this, they do not immediately give rise to morphisms of
varieties. However, we will see later (\cf\ \infra, Theorem~\ref{thm:main})
that they can be used to construct functions that are honest morphisms, and
that even admit expressions in the Dixmier--Ohno invariants.

\begin{remark}
  It is in fact possible to get integral matrices $T$ of smaller homogeneous
  weight. The best method known to us multiplies $T_0$ by a matrix over a
  larger extension whose determinant is homogeneous of degree $6$ in the
  coefficients. We do not need this more effective integral version in what
  follows.
\end{remark}

\subsection{Extension from $\Zs$ to $\Ys$}

In this section we describe how Proposition~\ref{prop:rationality} can be
extended from $\Zs$ to all of $\Ys$. Let $T$ be a universal transformation (as
defined in the previous section) that sends $\rhob$ to a non-zero multiple of
$v_2^2 - v_1 v_3$. With the notation of Proposition~\ref{prop:rationality}, we
can then consider the functions $\jb \ell^* T$ and $\bb_0 \ell^* T$ on $\Ys$.

\begin{proposition}\label{prop:integrality}
  Let $\jb$ be a joint invariant of degree $d$. Then the functions $\jb \ell^*
  T$ and $\bb_0 \ell^* T$ on $\Ys$ satisfy the following properties:
  \begin{enumerate}
  \item the quotient $\jb \ell^* T / (\bb_0 \ell^* T)^d$ is $\GL
    (V)$-invariant and does not depend on the choice of $T$;
  \item for $T = T_0$ the functions $\jb \ell^* T$ and $\bb_0 \ell^* T$ are in
    $K (\Sym^4 (V^*))$;
  \item for $T = T_{\intt}$ the functions $\jb \ell^* T$ and
    $\bb_0 \ell^* T$ are in $K [\Sym^4 (V^*)]$.
  \end{enumerate}
\end{proposition}
\begin{proof}
  (i): Any two possible $T$ differ by an element of $\GO (v_2^2 - v_1 v_3)$,
  the image of $\GL (W)$ under $h^*$. The result follows because $\jb$ and
  $\bb_0$ are invariant under $\SL (W)$ and $\jb / \bb_0^d$ is homogeneous of
  weight $0$.

  (ii): The matrix $T_0$ is a priori defined over a quadratic extension of $K
  [\Sym^4 (V^*)]$, as are therefore the functions $\jb \ell^* T$ and $\bb_0
  \ell^* T$. By Galois theory it suffices to show that applying its quadratic
  conjugate $T_0^{\sigma}$ to the universal ternary quartic yields the same
  result. One calculates first of all that $\det (T_0^{\sigma}) = -\det (T_0)$.
  This means that $T_0$ and $T_0^{\sigma}$ differ by the product of an element
  of $\SO (v_2^2 - v_1 v_3)$ and the scalar matrix $-1$. Now $-1$ is central
  and its action on quartic forms is trivial. The conclusion therefore follows
  from the fact that $\jb$ and $\bb_0$ are invariant under $\SL (W)$.

  (iii): This follows by combining the argument in the proof of (ii) with the
  fact that the polynomial ring $K [\Sym^4 (V^*)]$ is integrally closed in its
  field of fractions $K [\Sym^4 (V^*)]$.
\end{proof}

\begin{remark}
  Note that in general $\jb \ell^* T$ and $\bb_0 \ell^* T$ contain superfluous
  factors and are far from being $\SL (V)$-invariant themselves.
\end{remark}

\subsection{A fundamental lemma}

Let $\Fb$ be the universal ternary quartic, and let $\rhob$ be the
contravariant defined in~\eqref{eq:rho}. As in the previous section, let $T$ be
an algebraic transformation matrix that sends $\rhob$ to the multiple $\ub\cdot
(v_2^2 - v_1 v_3)$, or more precisely such that we have
\begin{equation}\label{eq:rhoT}
  \rhob (\av, \vv . T) = \ub \cdot (v_2^2 - v_1 v_3)\,.
\end{equation}
We can then define $\tilde{\Fb} = \Fb . T = \Fb (\av, \xv . T) = \Fb (T . \av,
\xv)$. Since the
covariant $\rhob$ is homogeneous of weight equals $(4 \cdot 4 + 2)/3 = 6$,
Eq.~\eqref{eq:weight_contra} shows that
\begin{equation}
  \tilde{\rhob}
  :=
  \rhob (\tilde{\Fb})
  =
  \rhob (T . \av, \vv) 
  =
  \det(T)^6 \cdot \rhob (\av, \vv . T)
  =
  \det(T)^6 \,\cdot\, \ub \,\cdot\, (v_2^2 - v_1 v_3)\,.
\end{equation}
Using the definitions of the operators in Section~\ref{sec:dixmier}, we see that
\begin{equation}\label{eq:I9tilde}
  \begin{split}
    \tilde{\Tb} := \Tb(\tilde{\Fb}, v)
    =
    {D(\tilde{\rhob}, \tilde{\Fb})}/{12}
    =
    \det(T)^6 \cdot {D(\rhob (\av, \vv . T), \tilde{\Fb})}/{12}\,,
    \\
    \tilde{\Ib}_9 := \Ib_9(\tilde{\Fb})
    =
    J_{11}(\tilde{\Tb},\tilde{\rhob})
    =
    \det(T)^6 \cdot J_{11}(\tilde{\Tb}, \rhob (\av, \vv . T)) .
  \end{split}
\end{equation}

By using the particular shape of $\rhob (\av, \vv . T)$, we calculate that in
terms of the coefficients $\tilde{a}$ of $\tilde{\Fb}$ we have that
\begin{equation}
  \tilde{\Ib}_9
  =
  \det(T)^{12} \cdot \ub^2 \cdot
  \left( \frac{\tilde{a}_{202}}{6} - \frac{\tilde{a}_{121}}{6} +
    \tilde{a}_{040} \right) .
\end{equation}
Since $\Ib_9$ is of weight $9 \cdot 4/3 = 12$,~\eqref{eq:weight_contra} also
gives that
\begin{equation}
  \Ib_9
  =
  \det (T)^{-12} \cdot \tilde{\Ib}_9
  = \ub^2 \cdot \left( \frac{\tilde{a}_{202}}{6} - \frac{\tilde{a}_{121}}{6} +
    \tilde{a}_{040} \right) .
\end{equation}

On the other hand, using the equivariant isomorphism $\ell^*$ from
\eqref{eq:equivariant} we see that in terms of the coefficients $\tilde{\av}$
of $\tilde{F}$, the constant binary form $\bb_0 \ell^* T$ is given by
\begin{equation}
  \bb_0 \ell^* T
  =
  \frac{\tilde{a}_{202}}{30} - \frac{\tilde{a}_{121}}{30} + \frac{\tilde{a}_{040}}{5}\,.
\end{equation}
Thus we obtain the algebraic equality
\begin{equation} \label{eq:I9b0}
  \Ib_9 = 5 \cdot \ub^2 \cdot (\bb_0 \ell^* T)\,.
\end{equation}

Finally, from Section~\ref{sec:dixmier} we know that the discriminant of
$\rhob$ equals $\Ib_{12}$ whereas the discriminant of $\rhob (\av, \vv . T)$
equals $- \ub^3 / 4$ by~\eqref{eq:rhoT}. Since $\rhob (\av, \vv . T) = \rhob
(\av, T^{-1} \,.\, \vv)$, we see that the matrix $[ \rhob (\av, \vv . T) ]$
that represents the symmetric bilinear form corresponding to the quadric contravariant
$\rhob (\av, \vv . T)$ can be obtained from the matrix $[ \rhob ]$ that
represents $\rho$ as $T^{-1} [ \rhob ] (T^{-1})^{t}$. Taking the discriminant comes
down to taking the determinant of these forms, and therefore
\begin{align}\label{eq:ub3}
  \begin{split}
    - \ub^3 /4
    & = \disc (\rhob (\av, \vv . T))
    = \det ([ \rhob (\av, \vv . T)])
    = \det (T^{-1} . [ \rhob (\av, \vv) ] . (T^{-1})^{t}) \\
    & = \det (T^{-1})^2 \det ([ \rhob (\av, \vv) ])
    = \det (T^{-1})^2 \disc (\rhob (\av, \vv)) \\
    & = \Ib_{12} / \det (T)^2 .
  \end{split}
\end{align}
Since on $\Zs$ we can moreover take $T$ to be the identity,~\eqref{eq:I9b0} and
\eqref{eq:ub3} translate into the following result, which will be fundamental
to our study of the integrality of the functions $\jb \ell^* T / (\bb_0 \ell^*
T)^d$.

\begin{lemma}\label{lem:I9b0}
  We have the equalities
  \begin{equation}
    \Ib_9
    =
    5 \cdot \ub^2 \cdot (\bb_0 \ell^*)
  \end{equation}
  of functions on $\Zs$ and
  \begin{equation}
    \left(\frac{\Ib_9}{\bb_0 \ell^* T}\right)^3
    =
    2000\cdot\left(\frac{\Ib_{12}}{\det(T)^2}\right)^2
  \end{equation}
  of functions on $\Ys$.
\end{lemma}

\subsection{Denominator bounds}
\label{sec:denom}

It turns out that the results of the previous sections give us a way to bound
the denominators of the functions $\jb \ell^*$ on $\Zs$ from
Proposition~\ref{prop:rationality}. Moreover, we can describe the functions
$\jb \ell^* T / (\bb_0 \ell^* T)^d$ on $\Ys$ as well. As a first step, we prove
two lemmata.

\begin{lemma}\label{lem:ufd}
  We have the following.
  \begin{enumerate}
  \item Let $S = P / Q$ be a simplified fraction in $K (\Sym^4 (V^*))$. If $S$
    is invariant under $\SL (V)$, then so are the elements $P$ and $Q$ of $K
    [\Sym^4 (V^*)]$.
  \item Let $P \in K [\Sym^4 (V^*)]$. If $P$ is invariant under $\SL (V)$,
    then so are all its irreducible factors in the ring $K [\Sym^4 (V^*)]$.
  \item The ring $K [\Sym^4 (V^*)]^{\SL (V)}$ is a unique factorization
    domain.
  \end{enumerate}
\end{lemma}

\begin{proof}
  (i) (See also~\cite[Exercise~6.10]{dolgainv}.) Because the fraction $S$ is
  simplified, the divisors defined by $P$ and $Q$ are invariant under $\SL
  (V)$. Therefore the action of $\SL (V)$ is described by multiplication with
  a character. But the group $\SL (V)$ has no non-trivial character.

  (ii) Consider the locus $P = 0$ in the affine space over $\Sym^4 (V)$. This
  locus consists of a finite number of irreducible components. The
  corresponding stabilizers are subgroups of $\SL (V)$ of finite index. These
  subgroups give rise to the full tangent space at the origin. By the classical
  correspondence between Lie subgroups and Lie subalgebras, they therefore
  correspond to a connected Lie subgroup of $\SL (V)$ of the same dimension as
  the latter group. But since $\SL (V)$ is irreducible, it is certainly
  connected, so we see that they all coincide with $\SL (V)$. This shows that
  the irreducible components of the locus $P = 0$ are invariant under $\SL
  (V)$.

  The irreducible factors of $P$ generate the radical ideals of the
  corresponding irreducible components, and are characterized up to a non-zero
  scalar by this property. Since the action of $\SL (V)$ preserves degree, we
  see that in fact these factors transform by a character under the action of
  $\SL (V)$. We can therefore again conclude by using the fact that $\SL (V)$
  has no non-trivial character.

  Part (iii) is a consequence of part (ii).
\end{proof}

\begin{remark}
  The fact that $K [\Sym^4 (V^*)]^{\SL (V)}$ is a unique factorization domain
  was known to Hilbert~\cite[II.3]{hilbert-inv}. General results in this
  direction (proved by different methods than those used in
  Lemma~\ref{lem:ufd}) were also obtained by Popov~\cite{popov-ufd} and
  Hashimoto~\cite{hashimoto-ufd}.
\end{remark}

\begin{lemma}\label{lem:I12}
  The element $\Ib_{12}$ of $K [\Sym^4 (V^*)]$ is irreducible.
\end{lemma}
\begin{proof}
  By Lemma~\ref{lem:ufd}\,-\,(i), it suffices to show irreducibility of
  $\Ib_{12}$ in the ring of Dixmier--Ohno invariants $K [\Sym^4 (V^*)]^{\SL
    (V)}$, and one can indeed verify by linear algebra and interpolation (\cf\
  Section~\ref{sec:interpolation}) that no expression of $\Ib_{12}$ in terms of
  invariants of lower homogeneous degree exists.
\end{proof}

Now we can state the following stronger version of
Proposition~\ref{prop:rationality}.

\begin{theorem}\label{thm:main}
  Let $\jb$ be a joint invariant of total degree $d$. With the notation of
  Section~\ref{sec:GHsec}, on $\Zs$ we have
  \begin{equation}\label{eq:mainaff}
    \jb \ell^*
    =
    \frac{\Pb}{\ub^{2 d}} ,
  \end{equation}
  where $\Pb$ is a polynomial in the Dixmier--Ohno invariants that is
  homogeneous of degree $9 d$. Moreover, we have an equality
  \begin{equation}\label{eq:mainproj}
    \frac{\jb \ell^* T}{(\bb_0 \ell^* T)^d}
    =
    \frac{\Pb}{\Ib_9^d}
  \end{equation}
  of functions on $\Ys$.
\end{theorem}
\begin{proof}
  Consider $(\jb \ell^*)^3$ as a function on $\Ys$. Then by
  Proposition~\ref{prop:rationality} we have $(\jb \ell^*)^3 = \Pb^3 /
  \Ib_{12}^n$ for some $n$ and hence it is invariant by $\SL(V)$ 
  Now the
  expression $(\jb\ell^*)^3$ is of homogeneous degree $3 d$ in the
  coefficients, it is homogeneous of weight $4 d$. Therefore by
  Lemma~\ref{lem:I9b0} we have
  \begin{equation}\label{eq:quotellT}
    (\jb \ell^* T)^3
    = \det (T)^{4 d} (\jb \ell)^3
    = \det (T)^{4 d} \frac{\Pb}{\Ib_{12}^n}
    = \frac{\Pb}{((\Ib_9 / \bb_0 \ell^* T)^{3 d} \Ib_{12}^{n - 2d})} .
  \end{equation}
  We already know $\jb \ell^* T / (\bb_0 \ell^* T)^d$ to be rational from
  Proposition~\ref{prop:integrality}\,-\,(i). Now we take $T = T_{\intt}$, so that
  both $\jb \ell^* T_{\intt}$ and $\bb_0 \ell^* T_{\intt}$ are polynomials. By
  Lemma~\ref{lem:ufd}, we can then rewrite~\eqref{eq:quotellT} as a polynomial
  equality in the unique factorization domain $K [\Sym^4(V^*)]$, where it
  becomes
  \begin{equation}
    \Ib_9^{3d} \cdot (\jb \ell^* T_{\intt})^3 \cdot \Ib_{12}^{n-2d}
    =
    \Pb \cdot (\bb_0 \ell^* T_{\intt})^{3 d}.
  \end{equation}
  By Lemma~\ref{lem:I12}, if we show that $\Ib_{12}$ does not actually divide
  the polynomial $\bb_0 \ell^* T$, then we get that $n \leq 2d$. One verifies
  this by finding a single plane quartic $F$ for which $\bb_0 \ell^* T_{\intt}
  (F) \neq 0$ while $\Ib_{12} (F) = 0$, which can be done by using a computer
  algebra system. This proves that one has
  \begin{equation}
    \frac{\jb \ell^* T}{(\bb_0 \ell^* T)^d}
    =
    \frac{\Pb \cdot\Ib_{12}^{s}}{\Ib_9^d}
  \end{equation}
  with $s$ non-negative. This yields the second statement of the theorem on
  $\Ys$. Transforming back, one obtains the first statement on $\Zs$.
\end{proof}

\begin{remark}
  As we see, the functions $\jb \ell^*$ themselves do not admit a rational
  expression in the Dixmier--Ohno invariants in $\Zs$ and therefore do not
  extend to an invariant function on $\Ys$. The group-theoretic reason for
  this is that under the action of $\SL (V)$ the normalization of $\rhob$ is
  only determined up to a factor $\zeta_3$, making the function $\ub$
  ill-defined on $\Ys$; only $\ub^3$ extends to a function on $\Ys$, and with
  it $(\jb \ell^*)^3$. On $\Zs$, where $\ub$ can be immediately recovered from
  the quartic by taking its covariant, this problem does not occur by
  construction.
\end{remark}

\begin{remark}
  The functions $\jb \ell^* T / (\bb_0 \ell^* T)^d$ also extend to $\Ys$
  because the characters of the numerator and the denominator under the action
  of $\langle \zeta_3 \rangle$ cancel. Note that for our reconstruction
  purposes we can use either of these formulas; the mutual quotients of the
  joint invariants suffice because geometrically we are only interested in
  quartics up to a scalar.
\end{remark}

\subsection{A projective approach}

It is possible to obtain the second result of Theorem~\ref{thm:main} without
intervention of the group $\langle \zeta_3 \rangle$, at least if one is willing
to work with projective spaces.

In order to achieve this, we use the inclusion $\P \Zs \subset \P \Ys$ instead,
which is a $(G, H)$-section for the group morphism $h^* : \PSL (W) \to \PSL
(V)$ with image $\SO(V)$. Using the equivariant isomorphism
\eqref{eq:equivariant} and the first part of Proposition~\ref{prop:section}, we
see that $(\jb \ell^*)/(\bb_0 \ell^*)^d = \Pb/\Qb \in K(\P \Ys)^G$, where $\Pb$
and $\Qb$ are polynomials in the Dixmier--Ohno invariants that are homogeneous
of the same degree.

We now shrink $\P \Ys$ to $\P' \Ys = \P \Sym^4(V^*) \setminus (\{\Ib_{12}=0\}
\cup \{\Ib_9=0\})$ and define $\P' \Zs$ correspondingly. By
Lemma~\ref{lem:I9b0} the function $\bb_0 \ell^*$ is non-zero on the
representatives of elements of $\P' \Zs$. Since $\P' \Ys$ is normal and affine
we can use the second part of Proposition~\ref{prop:section} to prove the
following lemma.
\begin{lemma}\label{lem:sandt}
  There exist $s, t \in \N$ and a polynomial $\Pb$ in the Dixmier--Ohno that is
  homogeneous of degree $12 s + 9 t$ such that
  \begin{equation}
    \frac{\jb \ell^*}{(\bb_0 \ell^*)^d} = \frac{\Pb}{\Ib_{12}^s \cdot \Ib_9^t} .
  \end{equation}
  on $\P' \Zs$.
\end{lemma}

If we use $T = T_{\intt}$, we have proved in Proposition~\ref{prop:integrality}
that $\jb \ell^* T$ and $\bb_0 \ell^* T$ belong to $K[\Sym^4(V^*)]$. Moreover
\begin{equation}
  \frac{\jb \ell^* T}{(\bb_0 \ell^* T)^d}
  =
  \frac{\Pb.T}{(\Ib_{12}.T)^s \cdot (\Ib_9.T)^t}
  =
  \frac{\Pb}{\Ib_{12}^s \cdot \Ib_9^t}
\end{equation}
by invariance of the denominator and numerator under $\SL(V)$ and the fact that
they have the same degree. Hence, in the unique factorization domain
$\C[\Sym^4(V^*)]$ we have the equality of polynomials
\begin{equation}
  (\jb \ell^* T_{\intt}) \cdot (\Ib_{12}^s \cdot \Ib_9^t)
  = \Pb \cdot ( \bb_0 \ell^* T_{\intt})^d.
\end{equation}
As before, one verifies that $\bb_0 \ell^* T_{\intt}$ is not divisible by
$\Ib_{12}$ by finding a single plane quartic $F$ for which $\bb_0 \ell^*
T_{\intt} (F)$ is non-zero while $\Ib_{12}(F)=0$. Hence $\Ib_{12}^s$ divides
$\Pb$. If we denote $\Pb_1 = \Pb/\Ib_{12}^s$, since $\Pb$ and $\Ib_{12}^s$ are
invariants under the action of $\SL (V)$, $\Pb_1$ is as well and is therefore
a polynomial in the Dixmier--Ohno invariants. We can also assume, after a
possible division, that $\Pb_1$ is coprime to $\Ib_9$ and we can write $(\jb
\ell^*)/( \bb_0 \ell^*)^d = \Pb_1/\Ib_9^{t_1}$ with $t_1 \in \N$.

To conclude, we only have to prove that $t_1 = d$. Again we have
\begin{equation}
  \frac{\jb \ell^* T_{\intt}}{(\bb_0 \ell^* T_{\intt})^d}
  =
  \frac{\Pb_1}{\Ib_9^{t_1}}\,,
\end{equation}
which gives
\begin{equation}\label{eq:9left}
  \jb \ell^* T_{\intt}
  =
  \frac{\Pb_1}{\Ib_{9}^{t_1}} \cdot (\bb_0 \ell^* T_{\intt})^d.
\end{equation}
From Lemma~\ref{lem:I9b0} and~\eqref{eq:elem} we obtain that, up to a
multiplicative constant,
\begin{equation}
  \bb_0^{3d} \ell^* T_{\intt}
  = 
  \Ib_9^{3d} \frac{\det(T_{\intt})^{4d}}{\Ib_{12}^{2d}}
  = 
  \Ib_9^{3d} \cdot  a^{30 d} \delta^{24d}
  \left(\frac{\Deltab}{\Ib_{12}}\right)^{2d}
  =  \Ib_9^{3d}  \cdot a^{30 d} \delta^{24d}
\end{equation}
since $\Deltab = \disc (\rhob) = \Ib_{12}$. We therefore see that
\eqref{eq:9left} translates, up to a constant, into
\begin{equation}
  \Ib_9^{3 t_1} \cdot (\jb \ell^* T_{\intt})^3
  =
  \Ib_9^{3d} \cdot \Pb_1^3 \cdot  (a^{30 d} \delta^{24d})
\end{equation}
as an equality in $\C[\Sym^4(V^*)]$. To conclude, let us observe that $\Ib_9$
is irreducible. Indeed, otherwise it would decompose as a product of invariants
by Lemma~\ref{lem:ufd}\,-\,(i) and therefore it would be a polynomial in $\Ib_3$
and $\Ib_6$. This is not possible as $\Ib_9, \Ib_6, \Ib_3$ are algebraically
independent. Moreover $a$ and $\delta$ are polynomials of degree less than
$\Ib_9$, so $\Ib_9$ is coprime to $a$ and $\delta$ and by hypothesis to
$\Pb_1$. Therefore we get $t_1 = d$, which concludes the proof.

\subsection{Evaluation-interpolation}
\label{sec:interpolation}

Since the 63 joint invariants given in Table~\ref{tab:jinv} form a minimal
basis of the ring of invariants $K [\Sym^8 (W^*) \oplus \Sym^4 (W^*)]^{\SL
  (W)}$, we proceed to determine the polynomials $\Pb$ in
Theorem~\ref{thm:main} for them.
Here general algebraic-geometric considerations do not seem to help us further,
and we therefore use a method that has served us well in the past, namely that
of evaluation-interpolation~\cite{LR11}.

From~\eqref{eq:I9b0} and~\eqref{eq:mainproj}, we seek to determine the regular
function on $\Ys$ defined by $\jb\,\ell^*\,T_0 \times \Ib_9^d\, /\,
(\bb_0\,\ell^*\,T_0\,)^d$ as an homogeneous polynomial $\Pb$ in the
Dixmier--Ohno invariants; we may as well restrict to $\Zs$ because up to the
usual scalar the calculated polynomial $\Pb$ is identical.
Our strategy is to evaluate $\Pb$ at lots of random quartics $F$ in $\Ys$, so
that we get equations
\begin{equation}\label{eq:eval}
  \Pb (\Ib_3(F), \Ib_6(F), \ldots, \Ib_{27}(F))
  = \jb (\ell^*(\tilde{F})) \times
  \frac{\Ib_9 (F)^d}{\bb_0 (\ell^*(\tilde{F}))^d} .
\end{equation}
Since we know the homogeneous weight $9 d$ of $\Pb$, we can recover it by
evaluating it in a known finite number of points. More systematically, our
strategy is as follows:
\begin{enumerate}
\item Generate a large family of random quartics $F$;
\item Given a quartic $F$ in the family, use either the matrix~\eqref{eq:T0}
  (or~\eqref{eq:Tint}) from Section~\ref{sec:normalization} to determine a
  quartic $\tilde{F}$ with normalized covariant $\tilde{\rho}$\,;
\item Calculate the quotients $\jb (\ell^*(\tilde{F})) \, \Ib_9(F)^d \, / \,
  \bb_0 (\ell^*(\tilde{F}))^d$ as in~\eqref{eq:eval} by using our explicit
  knowledge of $\ell^*$ and the joint invariants on $\Sym^8 (W) \oplus
  \Sym^4 (W)$ from Section~\ref{sec:jointinvs};
\item Evaluate at $F$ all the monomials of degree $9d$ in
  the Dixmier--Ohno invariants;
\item Determine $\Pb$ by solving the linear equation in its coefficients
  obtained from the evaluations in (iii) and (iv).
\end{enumerate}

The interpolation in the final step amounts to a search for a polynomial that
is homogeneous of degree $9 d$ in the Dixmier--Ohno invariants such that its
values coincide with those of $\jb (\ell^*(\tilde{F})) \, \Ib_9^d (F) \, / \,
\bb_0(\ell^*(\tilde{F}))^d$.
The most expensive computation in this approach is the matrix inversion, the
complexity of which depends of the dimension of the matrix, \ie\ the number of
monomials of degree $9d$ in the Dixmier--Ohno invariants. This number is equal
to the coefficient of $x^{9d}$ in the series expansion of $\prod
(1-x^{\deg(\Ib)})^{{-}1}$ where the $\Ib$ are the Dixmier--Ohno invariants. In
the largest computations that we have done, we had $d = 10$, which gives a
matrix of dimension 27398. The size of the required family of quartics
should be larger than this dimension; taking one hundred more is enough
in practice.

An additional complication is that the Dixmier--Ohno invariants satisfy certain
relations. This means that in order for $\Pb$ to be unique we have to restrict
to a subset of the monomials considered above, the size of which is given by a
coefficient of the Hilbert series of the ring of Dixmier--Ohno invariants,
which was determined in~\cite{shioda67}. While this coefficient can indeed be
computed, determining a correct subset amounts to calculating a Gröbner basis
of the ideal of relations. This still seems to be beyond reach. We therefore
run into the additional complication that the polynomial $\Pb$ that we obtain
is not unique. For example, because of the relations between the Dixmier--Ohno
invariants the interpolation matrix already has a kernel of dimension $14861$
for $d = 10$.

Directly inverting the matrix over a field of characteristic $0$ turns out to
be a rather suboptimal approach, since one then quickly runs out of memory.
Moreover, the heights of the corresponding polynomial expression also
increases rapidly, with a numerator and denominator of about $270$ decimal
digits when $d = 10$. To ease our computations, we have therefore used another
classical trick, namely the determination of the coefficients of $\Pb$ modulo
a sufficiently large amount of 9-digit primes. Using the Chinese remainder
theorem then allowed us to recover a hypothetical solution over $\Q$.

Once a conjectural expression for $\Pb$ has been obtained, the result can be
verified to be correct \emph{a posteriori}, that is, by checking that it gives
the correct answer on a sufficiently large set of ternary quartics. More
precisely, our procedure to verify the conjectural expression of a joint
invariant $\jb$ of degree $d$ is as follows.
\begin{enumerate}
\item Calculate the monomials $B$ of weight $9 d$ in the Dixmier--Ohno
  invariants;
\item Generate a large family $S$ of random quartics $F$;
\item Verify that the interpolated forms over $\Q$ are correct for the
  quartics in $S$, in the sense that~\eqref{eq:eval} holds for all $F$ in
  $S$; if not, terminate and indicate that the interpolation is incorrect;
\item Form the matrix $M$ obtained by evaluating the monomials in $B$ at the
  quartics in $S$;
\item If the rank of $M$ equals the $9 d$-th coefficient of the Hilbert series
  of the ring of Dixmier--Ohno invariants, terminate and return that the
  interpolation is correct; otherwise add more quartics to $S$ and start
  again at step (ii). 
\end{enumerate}
An especially pleasant feature of this verification algorithm is that it again
suffices to determine the rank of the matrix $M$ over a field of small
characteristic; this greatly speeds up the calculations.

\begin{table}[htbp]
  \centering
  \begin{tabular}{c|c|c|r|r}
    Degree $d$  & Dimension  & \# digits & Timings & Size (bytes)\\
    \hline\hline
    2      &   19    &  28    & 0 sec & $\simeq$ 50 b\\
    3      &   67    &  46    & 1 sec & $\simeq$ 1 kb\\
    4      &  206    &  64    & 5 sec & $\simeq$ 2,5 kb\\
    5      &  557    &  91    & 20 sec & $\simeq$ 10 kb\\
    6      & 1380    &  181    & 3,5 min & $\simeq$ 35 kb\\
    7      & 3166    &  181    & 23 min & $\simeq$ 80 kb\\
    8      & 6835    &  181    & 2,6 hours & $\simeq$ 200 kb\\
    9      & 13993   &  271    & 27 hours & $\simeq$ 500 kb\\
    10     & 27398   & 271  & 8 days & $\simeq$ 1 Mb\\
    11     & 51566   & ?    & ? & ?\\\hline
  \end{tabular}
  \caption{Interpolation timings}
  \label{tab:dimdig}
\end{table}

We performed our calculations with the computer algebra system
\textsc{magma}~\cite{Magma} for 62 out of the 63 invariants given in
Table~\ref{tab:jinv}. The exception is the invariant of largest degree
$\jb_{10, 1}$. Table~\ref{tab:dimdig} gives some indication of the accuracy
required in our computations. The timings in it are for a single invariant of
degree $d$, on a computer based on a single \textsc{intel} i7 -- 2.80 GHz
processor. The sign ``?'' means that the computation did not finish within the
time allotted to it.

As a result, we found for instance that
\begin{eqnarray*}
  7^2\times  \jb_{0,2} \cdot \Ib_9^2 / \bb_0^2
  &=&
  100\ \Ib_{9}^2 - 300\ \Ib_{18} \,,\\
  7^3\times \jb_{0,3} \cdot \Ib_9^3 / \bb_0^3
  &=&
  -1000\ \Ib_{9}^3+4500\ \Ib_{9}\, \Ib_{18}-13500\ \Ib_{12}\, \Ib_{15}\,,
\end{eqnarray*}
\begin{multline*}
  2^{5}\cdot 3^{4}\cdot 7^{2}\times{\jb_{2,0}}\cdot \Ib_9^2/\bb_0^2= 2\cdot
  5\cdot 7\cdot 13\, {{\Ib_{03}}}^{4}\, {\Ib_{06}}%
  -5\, {{\Ib_{03}}}^{3}\, {\Ib_{09}}\\%
  -3^{2}\cdot 5\, {{\Ib_{03}}}^{3}\, {\Jb_{09}}%
  +2^{6}\cdot 5\cdot 7\cdot 29\, {{\Ib_{03}}}^{2}\, {{\Ib_{06}}}^{2}%
  -2\cdot 5\cdot 23\, {{\Ib_{03}}}^{2}\, {\Ib_{12}}\\%
  +2^{2}\cdot 3\cdot 5\, {{\Ib_{03}}}^{2}\, {\Jb_{12}}%
  -2^{4}\cdot 3^{2}\cdot 5\cdot 17\, {\Ib_{03}}\,{\Ib_{06}}\,{\Ib_{09}}%
  +2^{4}\cdot 3^{3}\cdot 5^{2} \, {\Ib_{03}}\,{\Ib_{06}}\,{\Jb_{09}}\\%
  +2^{6}\cdot 5\cdot 7\, {\Ib_{03}}\,{\Jb_{15}}%
  + 2^{9}\cdot 3^{2}\cdot 5^{2}\cdot 7\, {{\Ib_{06}}}^{3}%
  + 2^{5}\cdot 3^{2}\cdot 5\cdot 113\, {\Ib_{06}}\,{\Ib_{12}}\\%
  -2^{6}\cdot 3^{2}\cdot 5^{2}\, {\Ib_{06}}\,{\Jb_{12}}%
  +2^{5}\cdot 3^{2}\, {{\Ib_{09}}}^{2} %
  -2^{3}\cdot 3^{2}\cdot 5\, {\Ib_{18}}%
  -2^{3}\cdot 3^{2}\cdot 5 \cdot 7 \, {\Jb_{18}}\,.
\end{multline*}

We observe that only very few primes divide the denominators of the
coefficients, namely 2, 3 and 7. In fact, for the 62 relations that we
computed, only 11 primes occur in the denominators of the coefficients; these
are 2, 3, 5, 7, 11, 13, 17, 23, 37, 79 and 89.

\begin{remark}
  It is perhaps worth noting 
  here that we do not need the full theoretical knowledge from
  Section~\ref{sec:denom} to obtain an inroad into our problem. Indeed,
  Proposition~\ref{prop:integrality} gives us a bound on the denominator of
  $\jb (\ell^* (\tilde{F})) / ({\bb_0} (\ell^* (\tilde{F}))^d$ since by
  Lemma~\ref{lem:ufd}\,-\,(ii) this denominator divides ${\bb_0} (\ell^*
  (\tilde{F}))^d$. In principle, this bound on the degree of the denominator
  gives us all that we need to make the above linear algebra work, at the cost
  of a longer running time since our lack of theoretical knowledge makes us
  miss the presence of certain cancellations between the numerator and
  denominator.
  
  Had we therefore lacked our particular knowledge of the denominator, we
  could, after finding some conjectural expressions in low degree, have
  proceeded under the \emph{assumption} that the invariants were is of the form
  $\Pb / \Ib_9^d$, derived conjectural expressions for the $\Pb$, and then
  verified that these are indeed the only rational expressions that satisfy the
  given denominator bound and coincide with $\jb \ell^* T / (\bb \ell^* T)^d$
  on a sufficiently large set of points. As this verification is relatively
  quick, this is a feasible approach. We include this remark because such an
  approach could still be useful in more general situations where fewer
  theoretical tools are available.
\end{remark}

\begin{remark}
  We did similar interpolation experiments for the locus of quartics on which
  either the covariant $\Tb$ or the covariant $\Xb$ of order 2 defined
  by~\eqref{eq:tau} are normalized. This time, we have to use the equivariant
  map $\ell$ instead of $\ell^*$ because we are using covariants instead of
  contravariants (\cf\ Section~\ref{sec:vr}).

  But a major issue is then that $\Tb$ and $\Xb$ are now of degree 5 rather
  than 4. This causes the role of the invariant $\Ib_9$ in~\eqref{eq:eval} to
  be taken over by invariants $\Ib_T$ and $\Ib_X$ of degree 21. While we have
  been able to determine the corresponding interpolation polynomials $\Pb$ for
  the joint invariants of degree $\leq$ 4, it now seems far more difficult to
  compute them for larger degrees. For $d = 10$, the largest degree in which
  we could successfully compute when using the contravariant $\rhob$, this
  leads to a linear system of dimension 16\,893\,297 over $\Q$.
\end{remark}

\section{Reconstruction}
\label{sec:reconstruction}

In this section we will explain how to reconstruct a quartic curve from its
Dixmier--Ohno invariants. After some theoretical considerations in
Section~\ref{sec:rec_locus}, we give a first version of our reconstruction
algorithms in Section~\ref{sec:rec_algs}. We briefly consider rationality
issues that appear when working over a non-algebraically closed base field in
Section~\ref{sec:descent}. The concluding Section~\ref{sec:implementation}
gives an example and discusses the efficiency of our current implementation of
these algorithms.

\subsection{Stability}
\label{sec:rec_locus}

In Section~\ref{sec:dixmier}, we have explained how to obtain the Dixmier--Ohno
invariants of ternary quartic forms. Geometrically, this gives rise to a
composed arrow
\begin{equation}\label{eq:arrow_aff}
  \Spec K [\Sym^4 (V^*)]
  \twoheadrightarrow
  \Spec K [\Sym^4 (V^*)]^{\SL (V)}
  \hookrightarrow
  \A^{13}_K .
\end{equation}
On $K$-points, this sends a ternary quartic form $F$ over $K$ to its
tuple of Dixmier--Ohno invariants. 

We will denote the closures of the image of~\eqref{eq:arrow_aff} by $\A_{\DO}$.
Now it is not always true that the inverse image of an element of $\A_{\DO}
(K)$ consists of a single $\SL (V)$-orbit of ternary quartic forms. If $F$ is a
ternary quartic, we will say that it is \defi{stable} if it belongs to the
stable locus under the action of $\SL (V)$ in the sense of geometric invariant
theory, \ie, if its orbit \emph{is} uniquely determined by its invariants. This
locus was characterized by Mumford (see~\cite[\S~4]{looijenga}); its geometric
points correspond to those quartics $F$ for which the curve $\Xs : F = 0$ is
reduced and has ordinary double points or cusps. This includes the locus of
ternary quartics $F$ that define smooth plane quartic curves $\Xs$.

As was sketched in the introduction, our way to reconstruct a ternary quartic
form from its Dixmier--Ohno invariants is essentially to make a
counterclockwise tour through the diagram
\begin{equation}
  \begin{split}
    % \leavevmode
    \xymatrix{
      \Spec K [\Sym^4 (V^*)] \supset \Zs
      \ar[r] \ar[d] &
      % \Zs' \subset \Spec K [\Sym^8 (W^*) \oplus \Sym^4 (W^*) \oplus \Sym^0 (W^*)]
      \Zs' \subset \Spec K [ \bigoplus_{n = 8,4,0} \Sym^n (W^*) ]
      \ar[d] \\
      \A_{\DO}
      \ar@{-->}[r]^{\phi} &
      \A_{\JI}
    }
  \end{split}
\end{equation}
Here the space $\A_{\JI}$ is the affine space corresponding to the generators
of the ring of joint invariants in Section~\ref{sec:jointinvs} along with the
invariant $\bb_0$; it is constructed as $\A_{\DO}$ above. The top arrow is
induced by the equivariant map $\ell^*$ from~\eqref{eq:equivariant}, the
vertical arrows are canonical projection maps, and the bottom arrow $\phi$ is
the rational map described by the functions in Theorem~\ref{thm:main}.

As we want to recover a quartic $F$ from its invariants, we assume that $F$ is
stable. Our method is also based on the assumption that $\Ib_{12}(F) \ne 0$. We
can therefore assume that $F \in \Zs (K)$ and consider the associated triple of
binary forms $\bv = (b_8, b_4, b_0) \in \Zs' (K) \subset \Ys' (K)$.

Since we have $\ub (F) \ne 0$ for such an $F$, the map $\phi$ is well-defined
by Proposition~\ref{prop:rationality}. We denote the joint invariants that are
the coordinates of the image of the Dixmier--Ohno invariants of $F$ by
$j_{d_1,d_2}$. To analyze when our reconstruction can work, it is then enough
to check when the $j_{d_1,d_2}$ determine a unique orbit of $\bv$ under
$\SL(W)$. Under the assumptions above, the following theorem proves that this
is always the case.

\begin{theorem}\label{thm:rec_locus}
  Let $F$ be a stable ternary quartic form such that $\Ib_{12}(F) \ne 0$ and
  let  $j_{d_1,d_2}$ the joint invariants image of the Dixmier--Ohno invariants
  of $F$ by $\phi$. Then the $j_{d_1, d_2}$ determine a unique triple $\bv =
  (b_8, b_4, b_0)$ up to the action of $\SL (W)$. 
\end{theorem}

As $b_0$ is a joint invariant, the proof comes down to determining when the
orbits of the pair $(b_8, b_4)$ are uniquely determined by their joint
invariants $\jb_{d_1,d_2}$. This is the case if and only if $\bv$ is in the
stable locus of the action of $\SL(W)$ on $\Sym_8(W^*) \oplus \Sym^4(W^*)$ in
the sense of geometric invariant theory. We have the following result.

\begin{lemma}
  An element $(b_8, b_4) \in \Sym^8 (W^*) \oplus \Sym^4 (W^*)$ is not stable
  for the action of $\SL (W)$ if and only if the forms $b_8$ and $b_4$ have a
  common root that is of multiplicity greater than $3$ for $b_8$ and of
  multiplicity greater than $1$ for $b_4$.
\end{lemma}

\begin{proof}
  This can be proved by using the numerical stability criterion due to Hilbert
  and Mumford. We follow the setting of~\cite[Sec.9.4]{dolgainv}. Considering
  the action of the maximal torus $\begin{bmatrix} t & 0 \\ 0  & t^{-1}
  \end{bmatrix}$ on $(b_8, b_4)$ shows that the set of weights of $(b_8, b_4)$ is
  contained in $S = \{-8, -6, \ldots, 6, 8\}$. More precisely it is equal to
  \begin{equation}
    wt (b_8,b_4)
    =
    \{ 2 i : \text{$i$ such that $b_{8,4+i} \ne 0$ or $b_{4,2+i} \ne 0$} \}\,,
  \end{equation}
  where we have written $b_8 = \sum b_{8,i} \, z_1^i z_2^{8-i}$ and $b_4 = \sum
  b_{4,i} \, z_1^i z_2^{4-i}$. By~\cite[Theorem~9.2]{dolgainv}, one knows that
  $(b_8, b_4)$ is stable for the action of this torus if and only if $0$
  belongs to the interior of the convex hull of $wt (b_8, b_4)$ in $\R$. Hence
  $(b_8, b_4)$ is not stable for this action if and only if 
  \begin{equation}
    \left\{\begin{array}{l}
        b_{8,8} = b_{8,7} = b_{8,6} = b_{8,5} = 0\,,\\
        b_{4,4} = b_{4,3} = 0\,,
      \end{array}\right.
    \; \text{ or } \;
    \left\{\begin{array}{l}
        b_{8,0} = b_{8,1} = b_{8,2} = b_{8,3} = 0 \,,\\
        b_{4,0} = b_{4,1} = 0\,.
      \end{array}\right.
  \end{equation}

  We denote the forms that satisfy one of these conditions by $U$. Now
  by~\cite[Theorem~9.3]{dolgainv}, we know that $(b_8,b_4)$ is not stable if and
  only if there exists a $T \in \SL(W)$ such that $(b_8,b_4).T \in U$. If
  $(b_8,b_4)$ is such that  $b_8$ and $b_4$ have a common root that is of
  multiplicity greater than $3$ for $b_8$ and greater than $1$ for $b_4$, we
  can find an element $T$ of $\SL(W)$ that sends this root to $(0:1)$, so that
  $(b_8, b_4).T \in U$. Conversely, if $(b_8,b_4)$ is not stable, then there
  exists $T \in \SL(W)$ such that $(b_8,b_4).T \in U$. This indeed implies that
  $b_8$ and $b_4$ have a common root with the specified multiplicities.
\end{proof}

\begin{proof}[Proof of Theorem~\ref{thm:rec_locus}]
  Suppose that $(b_8, b_4)$ is not stable. Then up to the action of $\SL (W)$,
  we have $b_8 = z_1^4 \cdot c_4 (z_1, z_2)$ with $c_4$ a universal form of
  degree $4$ and $b_4 = z_1^2 \cdot c_2 (z_1, z_2)$ with $c_2$ a universal form
  of degree $2$. The form $b_0$ can be chosen freely. We now consider the curve
  $F = (\ell^*)^{-1}(\bv)$, which depends on $9$ parameters. A computation with
  \Magma\ then shows that $F$ has a non-ordinary singularity at $(0 : 0 : 1)$.
\end{proof}

\subsection{Algorithms}
\label{sec:rec_algs}

Let $\Iv = (I_3, \dots, I_{27}) \in \A_{\DO} (K)$ be a tuple of Dixmier--Ohno
invariants, that is, an element of the image of the
embedding~\eqref{eq:arrow_aff}. The methods of the previous section allow us to
derive an algorithm to generically construct a ternary quartic form $F$ whose
tuple of Dixmier--Ohno invariants equals $\Iv$.

We want to use Theorem~\ref{thm:main}. To make that possible, we make our first
assumption.\smallskip

\begin{flushleft}
  \ \hspace*{1cm}(i) \parbox[t]{0.8\textwidth}{We have $I_{12} \neq 0$.}
\end{flushleft}\smallskip

\noindent
Because of this, we may suppose that $F$ is an element of $\Zs$ and  we can
then use~\eqref{eq:ub3} with $T=\textrm{Id}$ to find a value of the constant $u
\in K^*$. In light of~\eqref{eq:mainaff} we can then calculate the exact value
of any joint invariants $\jb \ell^* (F)$ by evaluating the polynomials $\Pb$
computed by interpolation in Section~\ref{sec:interpolation} at the tuple
$\Iv$. In the current implementation, we favor a different strategy in order to
avoid the computation of the cube root $u$. Since we know how the Dixmier--Ohno
invariants behave under scalar multiplication of the underlying forms, we can
further assume without loss of generality that $F$ satisfies $\bb_0\ell^* (F) =
I_9$. Beyond avoiding a cube root, this simplifies the
formula~\eqref{eq:mainproj}. If we are really interested in getting a form $F$
with Dixmier--Ohno invariants equal to $\Iv$ we will perform an extra step at
the end of our algorithm, as we will mention later.

In this setting we get $\jb_{2, 0}\ell^* (F)$, $\jb_{3, 0}\ell^* (F)$, \ldots
$\jb_{10, 0}\ell^* (F)$ as polynomial expressions in the $\Iv$. We now need to
construct $b_8$ and $b_4$ in the image
\begin{equation}\label{eq:rec_decomp}
  \ell^* (F)
  =
  (b_8, b_4, b_0)
  \in
  \Sym^8 (W^*) \oplus \Sym^4 (W^*) \oplus \Sym^0 (W^*) 
\end{equation}
with such joint invariants; note that $b_0$ is already given. We will rely on
the following generic strategy.  By Remark~\ref{rem:shioda}, we can compute the
Shioda invariants of the binary octic $b_8$ and we can then reconstruct $b_8$
by applying the results of~\cite{LR11}, at least as long as Shioda invariants
separate $\SL (W)$-orbits. We thus add this genericity assumption.\smallskip

\begin{flushleft}
  \ \hspace*{1cm}(ii) \parbox[t]{0.8\textwidth}{The roots of $b_8$ have
    multiplicities less than 4.}
\end{flushleft}\smallskip

\noindent
This assumption depends only on the invariants of $F$. More precisely, a
calculation shows that we have to be outside the locus defined by
\begin{equation}
  \begin{split}
    49\,{\jb^{2}_{3,0}}-81\,{\jb^{3}_{2,0}}=
    33\,\jb_{4,0}-25\,{\jb^{2}_{2,0}}=
    27\,\jb_{5,0}-20\,\jb_{3,0}\,\jb_{2,0}= 0\,, \\
    77\,\jb_{6,0}-50\,{\jb^{3}_{2,0}}=
    363\,\jb_{7,0}-125\,\jb_{3,0}\,{\jb^{2}_{2,0}}= 0\,,  \\
    \jb_{8,0}=
    \jb_{9,0}=
    \jb_{10,0} = 0\ \,
  \end{split}
\end{equation}

\begin{remark}
  We do not dig deeper into this matter here, but we plan to address the
  implementation of the full reconstruction problem in due time, especially to
  deal with the last assumption. For now one can treat the non-generic cases by
  using Gr\"obner basis calculations that make use of joint invariants with
  $b_4$.
\end{remark}

Now the results from~\cite{LR11} return an octic $b'_8$ with the given Shioda
invariants up to scalar multiplication in a weighted projective space, as there
is an action of $\SL (W)$ to play with. So there exists a scalar $\lambda$ such
that the invariants of $\lambda b'_8$ exactly equal the values $\jb_{d,0}
\ell^* (F)$. Such a constant $\lambda$ can be obtained from the evaluations at
$b'_8$ of the invariants $\jb_{d,0}$.\bigskip

Given $b_0 = I_9$ and $b_8 = \lambda\, b'_8$, it remains to determine $b_4$.
This determination simplifies under an additional generic assumption on the
$14$ joint invariants of degree~1 in the coefficients of $\bb_4$
\begin{equation}\label{eq:linforms}
  \jb_{2,1}, \jb_{3,1}, \jb_{4,1}, \jb'_{4,1}, \jb_{5,1}, \jb'_{5,1}, \jb_{6,1}, \jb'_{6,1},
  \jb_{7,1}, \jb'_{7,1}, \jb_{8,1}, \jb'_{8,1}, \jb_{9,1} \text{ and } \jb_{10,1}
\end{equation}
seen as linear forms in the coefficients of $\bb_4$, namely:

\begin{flushleft}
  \ \hspace*{1cm}(iii) \parbox[t]{0.8\textwidth}{The specialization of the linear
    forms~\eqref{eq:linforms} at $F$ generate a vector space of rank $5$.}
\end{flushleft}\smallskip

We can then find the coefficients of $b_4$ by solving a linear system.

\begin{remark}
  Experimentally, in the very rare cases where this assumption is not met, we
  complement these linear constraints with $5$ equations of degree $4$ that
  come from the relation
  \begin{equation}
    \rhob(\,(\ell^*)^{-1} ((b_8, b_4, b_0))\,) = u \cdot (v_2^2-v_1v_3)\,.
  \end{equation}
  Note that in addition to these 14+5=19 equations, we may make use of the
  joint invariants of degree 2, 3 and 4 in $\bb_4$ too. We have many equations
  of small degree in few unknowns. In such a situation, Gr\"obner basis
  calculations often yields the result.
\end{remark}

Now that we have all the factors of the decomposition~\eqref{eq:rec_decomp}, we
can determine a ternary quartic form $\tilde{F} = (\ell^*)^{-1} ((b_8, b_4,
b_0))$. By construction, this quartic has the requested Dixmier--Ohno
invariants seen as a point in the weighted projective space
$\P(1:2:3:3:4:4:5:5:6:6:7:7:9)$.

This suffices for our geometric purposes, but let us still illustrate how to
find a quartic form $F$ with Dixmier--Ohno invariants exactly equal $\Iv$, in
the simplest case where $I_3 \ne 0$. Compute $\lambda= I_3/\Ib_3(\tilde{F})$ and
let
\begin{equation}
  M =
  \left(\begin{matrix}
      \lambda & 0 & 0 \\
      0 & 1 & 0 \\
      0 & 0 & 1
    \end{matrix}\right).
\end{equation}
Since the weight of a degree $3 d$ Dixmier--Ohno invariant is $4d$ we see that
for $F = \frac{1}{\lambda} \cdot  \tilde{F}.M = \frac{1}{\lambda} \cdot
\tilde{F} (\lambda x_1, x_2, x_3)$ we get
\begin{equation}
  \Ib_3 (F)
  = \det(M)^4 \cdot \frac{\Ib_3(\tilde{F})}{\lambda^3}
  = \lambda \cdot \Ib_3(\tilde{F})
  = I_3.
\end{equation}

\subsection{Descent}
\label{sec:descent}

While the algorithm in the last section enables us to reconstruct generic
ternary quartic forms with given Dixmier--Ohno invariants, it has some
drawbacks that appear as soon as we try to work over a non-algebraically closed
base field $k$ of characteristic $0$. 

Under the genericity assumptions in the previous paragraph, the first
additional condition that we meet for reconstruction to be possible over the
base field is the existence of a binary octic $b'_8$ as constructed in that
paragraph with coefficients in $k$. That is, we have to assume that

\begin{flushleft}
  \hspace{1cm} (*) \parbox[t]{0.8\textwidth}{There exists a binary octic form
    $b'_8$ over $k$ whose tuple of Shioda invariants is equivalent to the
    Shioda invariants calculated from the given Dixmier--Ohno invariants.}
\end{flushleft}

If there exists such a $b'_8$ , then our main worry is to ensure that the
scalar $\lambda$ is $k$-rational. Since all invariants involved are rational,
an argument as in~\cite[\S~1.4]{LR11} is true under the following additional
generic assumption:

\begin{flushleft}
  \ \hspace*{1cm}(gen) \parbox[t]{0.8\textwidth}{The weights of the entries at
    which the Shioda invariants are non-zero forms a set whose elements
    generate the unit ideal in $\Z$.}
\end{flushleft}

The condition (gen) is generically satisfied. On the other hand, the condition
(*) is more subtle. Over number fields, it is in general impossible to avoid a
quadratic extension when applying the above approach directly (on the other
hand, such an obstruction never occurs over finite fields). Indeed, our
reconstruction always returns a ternary quartic $F$ such that the conic
associated to $\rhob (F)$ admits a rational point over its field of definition.
Hence, if we start from a quartic $F$ over $k$ without geometric automorphism
and such that $\rhob(F)$ has no $k$-rational point, such a quadratic extension
will necessarily occur.

So even when starting with a tuple of Dixmier--Ohno invariants that corresponds
to a quartic curve over $k$, the algorithms above may not return a model over
$k$ straight away; we usually need to pass to a quadratic extension. However,
by tinkering with Weil cocycles, we can often still obtain a model over $k$, as
is sketched in the following (artificial) example. We only scratch the surface
of the required theory; more information can be found in~\cite{LRRS} and in the
seminal work~\cite{wei56}.

\begin{example}\label{ex:desc}
  Let $\alpha = \sqrt{2}$ and consider the quartic curve $X$ defined by the
  polynomial $f$ given by
  \begin{small}
    \begin{equation}
      \begin{split}
        2 x_1^4 + (4 \alpha + 4) x_1^3 x_2 + x_1^3 x_3 + 18 x_1^2 x_2^2 + (2
        \alpha + 1) x_1^2 x_2 x_3 + x_1^2 x_3^2 + (8 \alpha + 4) x_1 x_2^3 \\
        + (2 \alpha + 2) x_1 x_2^2 x_3 + (\alpha + 1) x_1 x_2 x_3^2 + x_1
        x_3^3 + 5 x_2^4 + 2 x_2^3 x_3 + \alpha x_2^2 x_3^2 + \alpha x_2 x_3^3
        - x_3^4 .
      \end{split}
    \end{equation}
  \end{small}
  Then $X$ is isomorphic to its Galois conjugate $X^{\sigma}$ via the matrix
  \begin{equation}\label{eq:conj_iso}
    T =
    \begin{pmatrix}
      1 &  2 \alpha + 4 & 0 \\
      0 & -2 \alpha - 3 & 0 \\
      0 & 0 & 1
    \end{pmatrix}
  \end{equation}
  so that $f^{\sigma} = f . T$. Explicitly, the matrix $T$ can be found by
  similar methods as those developed in this paper, namely by normalizing both
  $f$ and $f^{\sigma}$ and thus reducing the problem to a question of $\SL
  (W)$-equivalence.

  The matrix $T$ satisfies the cocycle relation $T T^{\sigma} = 1$ in $\GL
  (V)$. This is not automatically the case, and one needs to consider the
  induced morphisms between spaces of differentials in general; \cf\ the
  analogous case for twists in~\cite{elg-twists}. By Hilbert's Theorem 90 we
  can find a matrix $T_0$ over $\Q (\alpha)$ such that $T = T_0^{-1}
  T_0^{\sigma}$, either by using a probabilistic method analogous to that
  employed in~\cite[\S~4.1]{LRRS} or writing out the entries of $T$ with
  respect to a $\Q$-basis of $\Q (\alpha)$ and solving the corresponding
  linear system. Regardless, a solution is given by
  \begin{equation}
    T_0 =
    \begin{pmatrix}
      1 & \alpha & 0 \\
      1 & 1 & 0 \\
      0 & 0 & 1
    \end{pmatrix} .
  \end{equation}
  We then have $(f . T_0^{-1})^{\sigma} = f^{\sigma} . T_0^{-\sigma} =
  f^{\sigma} . T^{-1} . T_0^{-1} = f . T_0^{-1}$, so that $f . T_0^{-1}$ is
  Galois invariant. And indeed
  \begin{equation}
    f . T_0^{-1} =
    x_1^4 + x_1^2 x_2 x_3 + x_1 x_2 x_3^2 + x_1 x_3^3 + x_2^4 - x_3^4 .
  \end{equation}
\end{example}

\begin{remark}
  In fact any smooth plane quartic with trivial geometric automorphism group
  descends to its field of moduli, as follows immediately from the criterion
  in~\cite{wei56}. So starting with a $k$-rational point of invariants $\Iv \in
  \P_{\DO} (k)$ of the projective space $\P_{\DO}$ corresponding to $\A_{\DO}
  (k)$, we can generically obtain a reconstruction over $k$ by first finding
  one over a quadratic extension by using our algorithms at the beginning of
  this section and then following the same method as in Example~\ref{ex:desc}.

  So if the geometric automorphism group of the curve corresponding to $\Iv$ is
  trivial, then we can reconstruct over the field of moduli $k$. When the
  characteristic of $k$ is either $0$ or large enough, then the results
  of~\cite{LRRS} show that we can also obtain a plane quartic over $k$
  corresponding to $\Iv$ if $\Iv$ corresponds to a plane quartic whose
  automorphism group has order strictly larger than $2$.

  It remains to discuss the case of smooth plane quartics with automorphism
  group $\Z / 2 \Z$. These admit a model over $K$ of the form
  \begin{equation}
    X : x_1^4 + f (x_2, x_3) x_1^2 + g (x_2, x_3) = 0 .
  \end{equation}
  It can be shown that if such a curve comes from a tuple $\Iv \in P_{\DO}
  (k)$, then the isomorphisms of $X$ with its conjugates induce a well-defined
  cocycle of automorphisms on the fixed line $x_1 = 0$, which is isomorphic to
  $\P^1$. This cocycle gives rise to a model $C$ of $\P^1$ over $k$, which is
  $k$-isomorphic to a conic over $k$. It is then possible to obtain a plane
  quartic over $k$ with invariants $\Iv$ if and only if the conic $C$ has a
  $k$-rational point. We hope to treat this subject more fully in future work;
  it is a generalization of the descent from $\C$ to its subfield $\R$
  considered in~\cite{artqui} and allows for a full classification of curves
  that do not descend to their field of moduli.
\end{remark}

\subsection{Implementation}
\label{sec:implementation}

We conclude this article by briefly discussing the implementation and
efficiency of our implementation. To do this, we give a single example of our
methods. We follow the procedure described in Section~\ref{sec:descent} for the
Dixmier--Ohno invariants
\begin{footnotesize}
  \begin{multline}\label{eq:dorat}
    \Iv = \left( %
      0,\, %
      0,\, %
      0,\, %
      0,\, %
      -{\frac { 7\cdot 19 }{ 2 ^{14}\cdot 3 ^{8}\cdot 5 ^{2}}},\, %
      0,\, %
      -{\frac { 11\cdot 19 }{ 2 ^{17}\cdot 3 ^{10}\cdot 5 ^{2}}},\, %
      0,\, %
      {\frac { 7\cdot 19 ^{2}}{ 2 ^{20}\cdot 3^{11}\cdot 5 ^{3}}},\, %
      {\frac { 19 ^{2}}{ 2 ^{20}\cdot 3^{11}\cdot 5 ^{3}}},\, \right.\\\left. %
      -{\frac { 19 ^{2}\cdot 31 }{ 2 ^{24}\cdot 3 ^{13}\cdot 5 ^{5}}},\, %
      -{\frac { 17\cdot 19 ^{2}}{ 2 ^{21}\cdot 3 ^{12}\cdot 5 ^{5}}},\, %
      -{\frac { 19 ^{2}\cdot 6553 }{ 2 ^{39}\cdot 3 ^{6}\cdot 5 ^{5}\cdot 11 }} %
    \right)\,.
  \end{multline}
\end{footnotesize}\noindent
We obtained this tuple by looking for a small rational point $(I_3$, $I_6$,
$\ldots$, $I_{27})$ in the space $\A^{\DO} (\Q) \cap \im \phi$ with $I_{12}
\neq 0$, $I_{27}\neq 0$. We therefore did not \textit{a priori} know a curve
with these invariants.

While these invariants are simple enough to write down, the intermediate
expressions calculated over $\Q$ in the algorithm have very large
coefficients. We therefore switch to the finite field $\F_{29}$ and perform
our computation there first.
We have $\Iv \bmod 29$ = $(0, 0, 0, 0, 17, 0, 17, 0, 2, 21, 4, 4, 9)$.
Applying Theorem~\ref{thm:main} by using our interpolation polynomial gives
us the joint invariants
\begin{equation}
  S (\Iv)  = (j_{2,0}, j_{3,0}, \ldots, j_{10,0}) = ( 4, 19, 4, 20, 19, 9, 24, 24, 8 ) .
\end{equation}
A corresponding binary octic $b_8$ is
\begin{equation}
  15\, y_1^8 + 9\, y_1^7 y_2 + 6\, y_1^6 y_2^2 + 19\, y_1^5 y_2^3 + 28\,
  y_1^4 y_2^4 + 16\, y_1^3 y_2^5 + 4\, y_1^2 y_2^6 + 25\, y_1 y_2^7 + 20\,
  y_2^8
\end{equation}
for which
\begin{equation}
  S (b_8) = (  7, 9, 5, 12, 18, 8, 23, 18, 11 ) .
\end{equation}
We have $S (b_8) = (\lambda^2 j_{2,0}, \ldots, \lambda^{10} j_{10,0})$ with
$\lambda = 26$. 

We have, $b_0 = I_9 = 0$ and we replace $b_8$ by $\lambda b_8 = 26 b_8$. It
remains to determine the coefficients of $b_4 = a_{4,0} y_1^4 + a_{3,1} y_1^3
y_2 + a_{2,2} y_1^2 y_2^2 + a_{1,3} y_1 y_2^3 + a_{0,4} y_2^4$. Using the
joint invariants, scaled by $\lambda$, leads to the system of equations
\begin{equation}
  \begin{split}
    13\, a_{4,0} + 21\, a_{3,1} + 4\, a_{2,2} + 17\, a_{1,3} + 25\, a_{0,4} + 9 = 0, \\
    9\, a_{4,0} + 8\, a_{3,1} + 10\, a_{2,2} + 19\, a_{1,3} + 25\, a_{0,4} + 18 = 0, \\
    9\, a_{4,0} + 22\, a_{3,1} + 7\, a_{2,2} + 18\, a_{1,3} + 22\, a_{0,4} + 17 = 0, \\
    22\, a_{4,0} + 15\, a_{3,1} + 3\, a_{2,2} + a_{1,3} + 19\, a_{0,4} + 18 = 0, \\
    23\, a_{4,0} + 10\, a_{3,1} + 19\, a_{2,2} + 28\, a_{1,3} + 22\, a_{0,4} + 26 = 0, \\
    13\, a_{4,0} + 25\, a_{3,1} + 8\, a_{2,2} + 12\, a_{1,3} + 3\, a_{0,4} + 16 = 0 .
  \end{split}
\end{equation}
This system admits the unique solution
\begin{equation}
  b_4 = 9\,y_1^4 + 8\,y_1^3\,y_2 + 19\,y_1^2\,y_2^2 + 17\,y_1\,y_2^3 + 23\,y_2^4
\end{equation}
and we get that the reconstruction $F = (\ell^*)^{-1} ((b_8, b_4, b_0))$
defines a curve given by
\begin{multline}
  24\, x_1^4 + 13\, x_1^3 x_2 + x_1^3 x_3 + 21\, x_1^2 x_2^2 + 22\, x_1^2 x_2
  x_3 + 28\, x_1^2 x_3^2 + 7\, x_1 x_2^3 + 23\, x_1 x_2^2 x_3 \\
  + 27\, x_1 x_2 x_3^2 + 10\, x_1 x_3^3 + 4\, x_2^4 + 24\, x_2^3 x_3 + 2\,
  x_2^2 x_3^2 + 20\, x_2 x_3^3 + 3\, x_3^4
\end{multline}
Over $\Q$, our algorithms returns on input~(\ref{eq:dorat}) a quartic with
coefficients that have around $100$ decimal digits. This inefficiency seems
to stem from the intermediate reconstruction from Shioda invariants. But by
using reduction theory (see the algorithms in~\cite{Stoll2011} and their
implementation in \Magma), we can obtain the much nicer quartic
\begin{small}
  \begin{equation}
    \begin{split}
      4\, x_1^4 - 12\, x_1^3 x_2 - 62\, x_1^3 x_3 - 108\, x_1^2 x_2^2 + 144\,
      x_1^2 x_2 x_3 + 12\, x_1^2 x_3^2 + 20\, x_1 x_2^3 \\
      - 90\, x_1 x_2^2 x_3 - 210\, x_1 x_2 x_3^2 + 125\, x_1 x_3^3 - 30\,
      x_2^4 - 160\, x_2^3 x_3 + 135\, x_2 x_3^3 + 180\, x_3^4
    \end{split}
  \end{equation}
\end{small}
Over a finite field the above coefficient explosion that we mentioned above is
not an issue, and our algorithms are correspondingly fast there. Their
performances is recorded in Table~\ref{tab:benchmarks} (average timings on an
\textsc{Intel} i7 -- 2.80 GHz processor with \Magma\ 2.21-10). Over finite
fields, it is not known to us when the Dixmier--Ohno invariants are generators
of the ring of invariants but we can at least verify \emph{a posteriori}
whether the constructed quartic has the correct Dixmier--Ohno invariants.

\begin{table}[h]
  \renewcommand*{\arraystretch}{1.2}
  \begin{tabular}{l|l|l|l|l|l}
    & $\F_p$ & $\F_{p^{25}}$ & $\F_{p^{50}}$ & $\F_{p^{75}}$ & $\F_{p^{100}}$ \\
    \hline
    \hline
    % $p = 19$ & 0.05 s & 0.6 s & 1.2 s & 1.2 s & 1.2 s \\
    % $p = 31$ & 0.05 s & 0.6 s & 0.8 s & 0.8 s & 1.3 s \\
    % 
    $p = 41$ & 0.05 s & 0.5 s & 0.8 s & 1.2 s & 2.8 s \\
    \hline
    \multicolumn{6}{c}{}\\[0.2cm]
    & $\F_p$ & $\F_{p^{2}}$ & $\F_{p^{3}}$ & $\F_{p^{4}}$ & $\F_{p^{10}}$ \\
    \hline
    \hline
    $p \approx 10^{50}$  & 0.2 s & 0.6 s & 0.9 s & 1.2 s & 3.1 s \\
    $p \approx 10^{100}$ & 0.2 s & 0.8 s & 1.1 s & 1.6 s & 5.9 s \\
    $p \approx 10^{150}$ & 0.3 s & 1.0 s & 1.6 s & 2.3 s & 9.8 s \\
    $p \approx 10^{200}$ & 0.4 s & 1.5 s & 2.6 s & 3.7 s & 18.4 s \\
    \hline
  \end{tabular}\medskip
  \caption{Reconstruction timings}
  \label{tab:benchmarks}
\end{table}

\end{document}